\documentclass[10pt]{amsart}

\usepackage{amsmath,amsfonts,amssymb}
\usepackage{graphicx,verbatim}
\usepackage[numbers,square,sort&compress]{natbib}
\usepackage{xspace}
\usepackage{xypic}
\usepackage{color}
\usepackage{a4wide}
\usepackage{tikz}
\usetikzlibrary{arrows}
\usetikzlibrary{decorations.markings}
\usetikzlibrary{patterns}
\usepackage{subfig}
\usepackage{enumerate}

\graphicspath{{./}{./figs/}{../../figs/}}

\theoremstyle{plain}
  \newtheorem{thm}{Theorem}[section]
  \newtheorem{prop}[thm]{Proposition}
  \newtheorem{lem}[thm]{Lemma}
  \newtheorem{cor}[thm]{Corollary}

\theoremstyle{definition}
  \newtheorem{defn}[thm]{Definition}
  \newtheorem{ex}[thm]{Example}
  \newtheorem{rem}[thm]{Remark}

\def\setmin{{-}}

\def\lift{\hat}

\def\R{\mathbb{R}}
\def\Z{\mathbb{Z}}

\def\AAA{\mathcal{A}}
\def\LLL{\mathcal{L}}

\DeclareMathOperator{\Acyc}{Acyc}
\DeclareMathOperator{\Hasse}{Hasse}
\DeclareMathOperator{\tor}{tor}
\DeclareMathOperator{\torHasse}{torHasse}
\DeclareMathOperator{\aff}{aff}

\DeclareMathOperator{\Chambers}{Cham}
\DeclareMathOperator{\Faces}{Face}
\DeclareMathOperator{\Pre}{Pre}
\DeclareMathOperator{\Cox}{C}
\DeclareMathOperator{\Conj}{Conj}
\DeclareMathOperator{\cl}{cl}

\def\into{\hookrightarrow}

\def\longto{\longrightarrow}

\def\epsilon{\varepsilon}

\def\ll{\prec}

\def\lleq{\preceq}

\def\<{\langle}
\def\>{\rangle}

\def\O{\mathcal{O}}
\def\P{\pi}

\def\p{p} 
\def\q{q} 
\def\r{r}

\def\quot{q}

\tikzset{big arrow/.style={decoration={markings,mark=at position 0.93
      with {\arrow[scale=2]{>}}}, postaction={decorate}, shorten >=8pt}}

\tikzset{big dashed arrow/.style={decoration={markings,mark=at position 0.93
      with {\arrow[scale=2]{>}}}, postaction={decorate}, shorten >=8pt, dashed}}

\newcommand\modone[1]{{{#1} \bmod 1}}

\begin{document}

\title{Morphisms and Order Ideals of Toric Posets}

\author[M.~Macauley]{Matthew Macauley} \address{Department of
  Mathematical Sciences \\ Clemson University \\ Clemson, SC 29634-0975, USA \newline\indent Department of Mathematics and Computer Science (IMADA) \\ University of Southern Denmark \\ \newline\indent DK-5230 Odense M}

\email{macaule@clemson.edu}

\thanks{The author was supported by Simons Foundation Collaboration Grant
  \#246042 and NSF grant DMS-1211691.}

\keywords{Coxeter group; cyclic order; cyclic reducibility; morphism; partial order; preposet; order ideal; order-preserving map; toric hyperplane arrangement; toric poset}

\begin{abstract} 
  Toric posets are in some sense a natural ``cyclic'' version of
  finite posets in that they capture the fundamental features of a
  partial order but without the notion of minimal or maximal
  elements. They can be thought of combinatorially as equivalence
  classes of acyclic orientations under the equivalence relation
  generated by converting sources into sinks, or geometrically as
  chambers of toric graphic hyperplane arrangements. In this paper, we
  define toric intervals and toric order-preserving maps, which lead
  to toric analogues of poset morphisms and order ideals. We develop
  this theory, discuss some fundamental differences between the toric
  and ordinary cases, and outline some areas for future
  research. Additionally, we provide a connection to cyclic
  reducibility and conjugacy in Coxeter groups. 
\end{abstract}

\maketitle

%%%%%%%%%%%%%%%%%%%%%%%%%%%%%%%%%%%%%%%%%%

%%=====================================================================
\section{Introduction}\label{sec:intro}
%%=====================================================================

A finite poset can be described by at least one directed acyclic graph
where the elements are vertices and directed edges encode
relations. We say ``at least one'' because edges implied by
transitivity may be present or absent. The operation of converting a source
vertex into a sink generates an equivalence relation on finite posets
over a fixed graph. Equivalence classes are called \emph{toric
  posets}. These objects have arisen in a variety of contexts in the
literature, including but not limited to chip-firing
games~\cite{ErikssonK:94}, Coxeter
groups~\cite{Eriksson:09,Shi:01,Speyer:09}, graph
polynomials~\cite{Chen:10}, lattices~\cite{Ehrenborg:09b, Propp:93},
and quiver representations~\cite{Marsh:03}. These equivalence classes
were first formalized as toric posets in~\cite{Develin:15}, where the
effort was made to develop a theory of these objects in conjunction
with the existing theory of ordinary posets. The name ``toric poset''
is motivated by a bijection between toric posets over a fixed
(undirected) graph and chambers of the toric graphic hyperplane
arrangement of that graph. This is an analogue to the well-known
bijection between ordinary posets over a fixed graph $G$ and chambers
of the graphic hyperplane arrangement of $G$, first observed by Greene \cite{greene1977acyclic} and later extended to signed graphs by Zaslavsky \cite{zaslavsky1991orientation}.

Combinatorially, a poset over a graph $G$ is determined by an \emph{acyclic orientation} $\omega$ of $G$. We denote the resulting poset by $P(G,\omega)$. A toric poset over $G$ is determined by an acyclic orientation, up to the \emph{equivalence} generated by converting sources into sinks. We denote this by $P(G,[\omega])$. Though most standard features of posets have elegant geometric interpretations, this viewpoint is usually unnecessary. In contrast, for most features of toric posets, i.e., the toric analogues of standard posets features, the geometric viewpoint is needed to see the natural proper definitions and to prove structure theorems. Once this is done, the definitions and characterizations frequently have simple combinatorial (non-geometric) interpretations. 

To motivate our affinity for the geometric approach, consider one of the fundamental hallmarks of an ordinary poset $P$: its binary relation, $<_P$. Most of the classic features of posets (chains, transitivity, morphisms, order ideals, etc.) are defined in terms of this relation. Toric posets have no such binary relation, and so this is why we \emph{need} to go to the geometric setting to define the basic features. Perhaps surprisingly, much of the theory of posets carries over to the toric setting despite the absence of a relation, and current toric poset research strives to understand just how much and what does carry over. As an analogy from a different area of mathematics, topology can be thought of as ``analysis without the metric.'' A fundamental hallmark of a metric space is its distance function. Many of the classic features of metric spaces, such as open, closed, and compact sets, and continuous functions, are defined using the distance function. However, once one establishes an equivalent characterization of continuity in terms of the inverse image of open sets, many results can be proven in two distinct ways: via epsilon-delta proofs, or topologically. When one moves from metric spaces to topological spaces, one loses the distance function and all of the tools associated with it, so this first approach goes out the window. Remarkably, much of the theory of real analysis carries over from metric spaces to topological spaces. Back to the poset world, one can prove theorems of ordinary posets using either the binary relation or the geometric definitions. However, upon passing to toric posets, the binary relation and all of the tools associated with it are lost, so one is forced to go to the geometric setting. Remarkably, much of the theory of ordinary posets still carries over to toric posets. This analogy is not perfect, because toric posets are not a generalization of ordinary posets like how topological spaces extend metric spaces. However, it should motivate the reliance on geometric methods throughout this paper.

An emerging theme of toric poset structure theorems, from both the original paper~\cite{Develin:15} and this one, is that characterizations of toric analogues, when they exist, usually have one of two forms. In one, the feature of a toric poset $P(G,[\omega])$ is characterized by it being the analogous feature of the ordinary poset $P(G,\omega')$, \emph{for all} $\omega'\in[\omega]$. In the other, the feature of $P(G,[\omega])$ is characterized by it being the analogous feature of $P(G,\omega')$ \emph{for some} $\omega'\in[\omega]$. Several examples of this are given below. It is not obvious why this should happen or which type of characterization a given toric analogue should have \emph{a priori}. Most of these are results from this paper, and so this list provides a good overview for what is to come.

\noindent The ``\emph{For All}'' structure theorems:
\begin{itemize}
\item A set $C\subseteq V$ is a toric chain of $P(G,[\omega])$ iff $C$ is a chain of $P(G,\omega')$ for all $\omega'\in[\omega]$. (Proposition~\ref{prop:toric-chains})
\item The edge $\{i,j\}$ is in the toric transitive closure of $P(G,[\omega])$ iff $\{i,j\}$ is in the transitive closure of $P(G,\omega')$ for all $\omega'\in[\omega]$. (Proposition~\ref{prop:Hasse})
\end{itemize}

\noindent The ``\emph{For Some}'' structure theorems:
\begin{itemize}
  \item A partition $\pi\in\Pi_V$ is a closed toric face partition of $P(G,[\omega])$ iff $\pi$ is a closed face partition of $P(G,\omega')$ for some $\omega'\in[\omega]$. (Theorem~\ref{thm:toric-face-partitions})
  \item A set $A\subseteq V$ is a (geometric) toric antichain of $P(G,[\omega])$ iff $A$ is an antichain of $P(G,\omega')$ for some $\omega'\in[\omega]$. (Proposition~\ref{prop:toric-antichains})
\item If a set $I\subseteq V$ is a toric interval of $P(G,[\omega])$, then $I$ is an interval of $P(G,\omega')$ for some $\omega'\in[\omega]$. (Proposition~\ref{prop:toric-intervals})
  \item A set $J\subseteq V$ is a toric order ideal of $P(G,[\omega])$ iff $J$ is an order ideal of $P(G,\omega')$ for some $\omega'\in[\omega]$. (Proposition~\ref{prop:toric-ideals})
  \item Collapsing $P(G,[\omega])$ by a partition $\pi\in\Pi_V$ is a morphism of toric posets iff collapsing $P(G,\omega')$ by $\pi$ is a poset morphism for some $\omega'\in[\omega]$. (Corollary~\ref{cor:toric-morphism})
\item If an edge $\{i,j\}$ is in the Hasse diagram of $P(G,\omega')$ for some $\omega'\in[\omega]$, then it is in the toric Hasse diagram of $P(G,[\omega])$. (Proposition~\ref{prop:Hasse})
\end{itemize}

This paper is organized as follows. In the next section, we formally
define posets and preposets and review how to view them geometrically
in terms of faces of chambers of graphic hyperplane arrangements. In
Section~\ref{sec:poset-morphisms}, we translate well-known properties
of poset morphisms to this geometric setting. In
Section~\ref{sec:toric-posets}, we define toric posets and preposets
geometrically in terms of faces of chambers of toric hyperplane
arrangements, and we study the corresponding ``toric face partitions'' and the bijection between toric preposets and
lower-dimensional faces. In
Section~\ref{sec:toric-intervals}, we define the notion of a toric
interval and review some features of toric posets needed for toric
order-preserving maps, or morphisms, which are finally presented in
Section~\ref{sec:toric-morphisms}. In Section~\ref{sec:toric-ideals},
we introduce toric order ideals and filters, which are essentially the
preimage of one element upon mapping into a two-element toric
poset. The toric order ideals and filters of a toric poset turn out to coincide. They form a graded poset $J_{\tor}(P)$, but unlike the ordinary case, this need not be a lattice. In Section~\ref{sec:coxeter}, we provide a connection of this theory to Coxeter groups, and then we conclude with a summary and discussion of current and future research in Section~\ref{sec:conclusion}.

%%=====================================================================
\section{Posets geometrically}\label{sec:posets}
%%=====================================================================

%%-------------------------------------------------------
\subsection{Posets and preposets}\label{subsec:preposets}

A \emph{binary relation} $R$ on a set $V$ is a subset $R\subseteq
V\times V$. A \emph{preorder} or \emph{preposet} is a binary relation
that is reflexive and transitive. This means that $(x,x)\in R$ for all
$x\in V$, and if $(x,y),(y,z)\in R$, then $(x,z)\in R$. We will use
the notation $x\lleq_R y$ instead of $(x,y)$, and say that $x\ll_Ry$
if $x\lleq_R y$ and $y\not\lleq_R x$. Much of the basics on preposets
can be found in~\cite{Postnikov:08}.

An \emph{equivalence relation} is a preposet whose binary relation is
symmetric. For any preposet $P$, we can define an equivalence relation
$\sim_P$ on $P$ by saying $x\sim_P y$ if and only if $x\lleq_P y$ and
$y\lleq_P x$ both hold. A \emph{partially ordered set}, or
\emph{poset}, is a preposet $P$ such that every $\sim_P$-class has
size $1$. We say that a preposet is \emph{acyclic} if it is also a
poset.

Every preorder $P$ over $V$ determines a directed graph $\omega(P)$
over $V$ that contains an edge $i\to j$ if and only if
$i\lleq_P j$ and $i\neq j$. Not every directed graph arises from such
a preorder, since edge transitivity is required. That is, if $i\to j$
and $j\to k$ are edges, then $i\to k$ must also be an edge. However,
any directed graph can be completed to a transitive graph via
transitive closure, which adds in all ``missing'' edges. Note that the
graph $\omega(P)$ is acyclic if and only if $P$ is a poset. The
strongly connected components are the $\sim_P$-classes, and so the
quotient $\omega(P)/\!\!\sim_P$ is acyclic and $P/\!\!\sim_P$ inherits
a natural poset structure from $P$.

If $R_1$ and $R_2$ are preorders on $V$, then we can define their
union $R_1\cup R_2$ as the union of the subsets $R_1$ and $R_2$ of
$V\times V$. This need not be a preorder, but its \emph{transitive
  closure} $\overline{R_1\cup R_2}$ will be.

Another way we can create a new preorder from an old one is by an
operation called contraction. Given a binary relation $R\subseteq
V\times V$, let $R^{op}$ denote the \emph{opposite} binary relation,
meaning that $(x,y)\in R^{op}$ if and only if $(y,x)\in R$. If $P$ and $Q$ are preposets on $V$, then $Q$ is a \emph{contraction} of $P$ if there is a binary relation $R\subseteq P$ such that $Q=\overline{P\cup R^{op}}$. Intuitively, each added edge $(x,y)\in R^{op}$ forces $x\sim_Q y$ because $(y,x)\in P\subseteq Q$ by construction. Note that in this context, contraction is a \emph{different} concept than what it often means in graph theory -- modding out by a subset of vertices, or ``collapsing'' a set of vertices into a single vertex.

%%--------------------------------------------------------------------
\subsection{Chambers of hyperplane arrangements}\label{subsec:chambers}

It is well known that every finite poset corresponds to a chamber of a
graphic hyperplane arrangement~\cite{Wachs:07}. This correspondence
will be reviewed here. Let $P$ be a poset over a finite set
$V=[n]:=\{1,\dots,n\}$. This poset can be identified with the
following open polyhedral cone in $\R^V$:
\begin{equation}
  \label{cone-of-a-poset-defn}
  c=c(P):=\{x\in\R^V:x_i<x_j\text{ if }i<_Pj\}.
\end{equation}
It is easy to see how the cone $c$ determines the poset $P=P(c)$: one
has $i <_P j$ if and only if $x_i<x_j$ for all $x$ in $c$. Each such
cone $c$ is a connected component of the complement of the graphic
hyperplane arrangement for at least one graph $G=(V,E)$. In this case,
we say that $P$ is a \emph{poset over $G$} (or ``on $G$''; both are
used interchangeably). Given distinct vertices $i,j$ of a simple graph
$G$, the hyperplane $H_{ij}$ is the set
\[
  H_{ij}:=\{x\in\R^V:x_i=x_j\}\,.
\]
The {\it graphic arrangement} of $G$ is the set $\AAA(G)$ of all
hyperplanes $H_{ij}$ in $\R^V$ where $\{i,j\}$ is in $E$. Under a
slight abuse of notation, at times it is convenient to refer to
$\AAA(G)$ as the set of points in $\R^V$ on the hyperplanes, as
opposed to the actual finite set of hyperplanes themselves. It should
always be clear from the context which is which.

Each point $x=(x_1,\ldots,x_n)$ in the {\it complement} $\R^V \setmin
\AAA(G)$ determines an {\it acyclic orientation} $\omega(x)$ of the
edge set $E$: direct the edge $\{i,j\}$ in $E$ as $i\to j$ if and only
if $x_i<x_j$. Clearly, the fibers of the mapping
$\alpha_G:x\longmapsto\omega(x)$ are the \emph{chambers} of the
hyperplane arrangement $\AAA(G)$. Thus, $\alpha_G$ induces a bijection
between the set $\Acyc(G)$ of acyclic orientations of $G$
and the set $\Chambers\AAA(G)$ of chambers of $\AAA(G)$:
\begin{equation}
  \label{eq:alpha-diagram1}
  \xymatrix{\R^V\setmin\AAA(G) \ar@{>>}[dr] \ar[rr]^{\alpha_G} & & \Acyc(G) \\
    & \Chambers\AAA(G) \ar@{.>}[ur] &\\ }
\end{equation}

We denote the poset arising from an acyclic orientation
$\omega\in\Acyc(G)$ by $P=P(G,\omega)$. An open cone $c=c(P)$ may be a
chamber in several graphic arrangements, because adding or removing
edges implied by transitivity does not change the
poset. Geometrically, the hyperplanes corresponding to these edges do
not cut $c$, though they intersect its boundary.  Thus, there are, in
general, many pairs $(G,\omega)$ of a graph $G$ and acyclic
orientation $\omega$ that lead to the same poset
$P=P(G,\omega)=P(c)$. Fortunately, this ambiguity is not too bad, in
that with respect to inclusion of edge sets, there is a unique minimal
graph $\hat{G}^{\Hasse}(P)$ called the \emph{Hasse diagram} of $P$ and
a unique maximal graph $\bar{G}(P)$, where $(G,\omega)\longmapsto
(\bar{G}(P),\bar{\omega})$ is \emph{transitive closure}.

Given two posets $P,P'$ on a set $V$, one says that $P'$ is an
\emph{extension} of $P$ when $i<_Pj$ implies
$i<_{P'}j$. Geometrically, $P'$ is an extension of $P$ if and only if
$c(P')\subseteq c(P)$.  Moreover, $P'$ is a \emph{linear extension} if
$c(P')$ is a chamber of $\AAA(K_V)$, where $K_V$ is the complete graph.

%%-------------------------------------
\subsection{Face structure of chambers}

Let $\pi=\{B_1,\dots,B_r\}$ be a partition of $V$ into nonempty
blocks. The set $\Pi_V$ of all such partitions has a natural poset
structure: $\P\leq_V\P'$ if every block in $\pi$ is contained in some
block in $\P'$. When this happens, we say that $\pi$ is \emph{finer}
than $\P'$, or that $\P'$ is \emph{coarser} than $\P$. 

Intersections of hyperplanes in $\AAA(G)$ are called \emph{flats}, and
the set of flats is a lattice, denoted $L(\AAA(G))$. Flats are
partially ordered by \emph{reverse inclusion}: If $X_1\subseteq X_2$,
then $X_2\leq_L X_1$. Every flat of $\AAA(G)$ has the form
\[
D_\pi:=\{x\in\R^V:\text{ $x_i=x_j$ for every pair $i,j$ in the same
  block $B_k$ of $\pi$}\}\,,
\]
for at least one partition $\pi$ of $V$. Note that $D_\pi\leq_L D_{\pi'}$ if and only if $\pi\leq_V\pi'$; this should motivate the convention of partially ordering $L(\AAA(G))$ by reverse inclusion.

Given a poset $P=P(G,\omega)$ over $G=(V,E)$, a partition $\pi$ of $V$ defines a preposet
$P_\pi:=(\pi,\leq_P)$ on the blocks, where $B_i\leq_P B_j$ for $x\leq_P
y$ for some $x\in B_i$ and $y\in B_j$ (and taking the transitive closure). This defines a directed graph $\omega/\!\!\sim_\pi$, formed by collapsing out each block $B_i$ into a single vertex. Depending on the context, we may use $P_\pi$ or $\omega/\!\!\sim_\pi$ interchangeably.
If this preposet is acyclic (i.e., if $P_\pi$ is a poset, or equivalently, the directed graph $\omega/\!\!\sim_\pi$ is acyclic), then we say that $\pi$ is \emph{compatible} with $P$. In this case, there is a canonical surjective poset morphism $P\to P_\pi$. We call such a morphism a \emph{quotient}, as to distinguish it from inclusions and extensions which are inherently different.

Compatibility of partitions with respect to a poset can be characterized by a closure operator on $\Pi_V$. If $P_\pi=(\pi,\leq_P)$ is a preposet that is not acyclic, then there is a unique minimal coarsening $\cl_P(\pi)$ of $\pi$
such that the contraction $(\cl_P(\pi),\leq_P)$ is acyclic. This
is the partition achieved by merging all pairs of blocks $B_i$ and
$B_j$ such that $B_i\sim_P B_j$, and we call it the \emph{closure}
of $\pi$ with respect to $P$. If $P$ is understood, then we may write
this as simply $\bar{\pi}:=\cl_P(\pi)$. A partition $\pi$ is \emph{closed} (with respect to $P$) if $\bar\pi=\pi$, which is equivalent to being compatible with respect to $P$. Geometrically, it means that for any $i\neq j$, there is some $x\in \overline{c(P)}\cap D_\pi$ such that $x_{b_i}\neq x_{b_j}$ for some $b_i\in B_i$ and $b_j\in B_j$. If $\pi$ is not closed, then the polyhedral face $\overline{c(P)}\cap D_\pi$ has strictly lower dimension than $D_\pi$. In this case, $\bar\pi$ is the unique coarsening that is closed with respect to $P$ and satisfies $\overline{c(P)}\cap D_{\bar\pi}=\overline{c(P)}\cap D_\pi$.

Still assuming that $P$ is a poset over $G=(V,E)$, and $\pi$ is a partition of $V$, define 
\begin{equation}\label{eq:F_pi}
  \bar{F}_\pi(P):=\overline{c(P)}\cap D_\pi\,.
\end{equation}
If $D_\pi$ is a flat of $\AAA(\bar{G}(P))$, then $\bar{F}_\pi(P)$ is a face of the (topologically) closed polyhedral cone
$\overline{c(P)}$. In this case, we say that $\pi$ is a \emph{face
  partition} of $P$. Since it is almost always clear what $P$ is, we will usually write $\bar{F}_\pi$ instead of $\bar{F}_\pi(P)$. If $D_\pi$ is not a flat of $\AAA(G)$, then the subspace $D_\pi$ still intersects $\overline{c(P)}$ in at least the line $x_1=\cdots=x_n$. Though this may intersect in the interior of $c(P')$, it is a face of $\overline{c(P')}\subseteq\overline{c(P)}$, for at least one extension $P'$ of $P$.
  
  To characterize the facial structure of the cone $\overline{c(P)}$, it suffices to characterize the closed face partitions. This is well known -- it was first described by Geissinger~\cite{Geissinger:81}, and also done by Stanley~\cite{Stanley:86} in the characterization of the face structure of the \emph{order polytope} of a poset, defined by
\begin{equation}\label{eq:order-polytope}
  \O(P)=\{x\in[0,1]^V: x_i\leq x_j\;\text{ if }\;i\leq_P j\}\,.
\end{equation}
Clearly, if $\pi$ is a closed face partition of $P$, then the
subposets induced by the individual blocks are connected (that is,
their Hasse diagrams are connected). We call such a partition
\emph{connected}, with respect to $P$.

\begin{thm}{(\cite{Stanley:86}, Theorem 1.2)}\label{thm:face-partitions}
  Let $P$ be a poset over $G=(V,E)$. A partition $\pi$ of $V$ is a
  closed face partition of $P$ if and only if it is connected and compatible with $P$.
\end{thm}

To summarize Theorem~\ref{thm:face-partitions}, characterizing the faces of $\O(P)$ amounts to characterizing which face partitions $\pi$ are closed. If $\pi$ is not compatible with $P$, then it is not closed. On the other hand, if $\pi$ not connected, then it is either not closed, or the flat $D_\pi$ cuts through the interior of $\O(P)$ (and hence of $c(P)$), in which case $\pi$ is not a face partition.

\begin{ex}
  Let $P$ be the poset shown at left in
  Figure~\ref{fig:face-partitions}; its lattice of closed face partitions is
  shown at right. In this and in later examples, we denote the blocks
  of a partition using dividers rather than set braces, e.g.,
  $\pi=B_1/B_2/\cdots/B_r$.

  The partition $\sigma=1/23/4$ is closed but not connected; it is not
  a face partition because $D_\sigma=H_{23}$ intersects the interior of
  $\overline{c(P)}$.  The partition $\pi=124/3$ is connected but not
  closed. Finally, the partition $\pi'=14/23$ is neither connected nor
  closed. However, both $\pi$ and $\pi'$ are face partitions because
  the subspaces $D_\pi$ and $D_{\pi'}$ intersect $\overline{c(P)}$ in
  the line $x_1=x_2=x_3=x_4$, which is the flat $D_V$. Therefore, both
  of these partitions have the same closure: $\cl_P(\pi)=\cl_P(\pi')=1234$.
  
  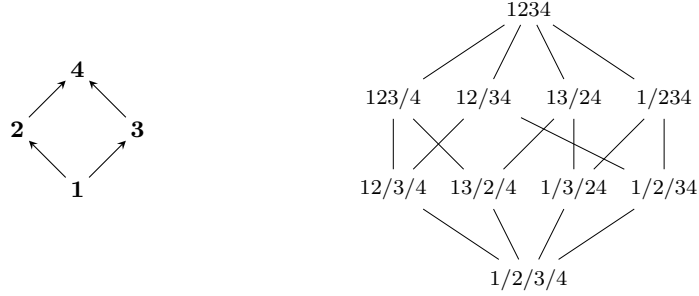
\begin{figure}\centering
    \tikzstyle{to} = [draw,-stealth]
    \begin{tikzpicture}
      \begin{scope}[shift={(-6,0)}, scale=.4, shorten >= -2.5pt, shorten <= -2.5pt]]
        \node (4) at (0,4) {\small $\mathbf{4}$};
        \node (2) at (-2,2) {\small $\mathbf{2}$};
        \node (3) at (2,2) {\small $\mathbf{3}$};
        \node (1) at (0,0) {\small $\mathbf{1}$};
        \draw[to] (1) to (2);
        \draw[to] (1) to (3);
        \draw[to] (2) to (4);
        \draw[to] (3) to (4);
      \end{scope}
      \begin{scope}[scale=.6]
      \node (1234) at (0,4) {\footnotesize $1234$};
      \node (123-4) at (-3,2) {\footnotesize $123/4$};
      \node (12-34) at (-1,2) {\footnotesize $12/34$};
      \node (13-24) at (1,2) {\footnotesize $13/24$};
      \node (1-234) at (3,2) {\footnotesize $1/234$};
      \node (12-3-4) at (-3,0) {\footnotesize $12/3/4$};
      \node (13-2-4) at (-1,0) {\footnotesize $13/2/4$};
      \node (1-3-24) at (1,0) {\footnotesize $1/3/24$};
      \node (1-2-34) at (3,0) {\footnotesize $1/2/34$};
      \node (1-2-3-4) at (0,-2) {\footnotesize $1/2/3/4$};
      \draw[very thin] (1234) -- (123-4) -- (12-3-4) -- (1-2-3-4)
      -- (1-2-34) -- (1-234) -- (1234);
      \draw[very thin] (1234) -- (12-34) -- (12-3-4);
      \draw[very thin] (1-234) -- (1-3-24) -- (13-24) -- (13-2-4) -- (123-4); 
      \draw[very thin] (1-2-34) -- (12-34);
      \draw[very thin] (1234) -- (13-24);
      \draw[very thin] (1-2-3-4) -- (13-2-4);
      \draw[very thin] (1-2-3-4) -- (1-3-24);
      \end{scope}
  \end{tikzpicture}
  %\end{center}
    \caption{A poset $P$ and its lattice of closed face partitions.}
    \label{fig:face-partitions}
  \end{figure}
\end{ex}

If $\bar\pi=\pi=\{B_1,\dots,B_r\}$ is a closed face partition of $P$,
then $D_\pi$ is an $r$-dimensional flat of $\AAA(G)$, and the closed face $\bar{F}_\pi$ is an $r$-dimensional subset of $D_\pi\subseteq\R^V$. The interior of $\bar{F}_\pi$ with
respect to the subspace topology of $D_\pi$ will be called an open
face. So as to avoid confusion between open and closed faces, and open
and closed chambers, we will speak of faces as being features of the
actual poset, not of the chambers. It should be easy to relate these
definitions back to the chambers if one so desires.

\begin{defn}
  A set $F\subseteq\R^V$ is a \emph{closed face} of the
  poset $P$ if $F=\bar{F}_\pi=\overline{c(P)}\cap D_\pi$ for some closed
  face partition $\pi=\cl_P(\pi)$ of $V$. The interior of $\bar{F}_\pi$
  with respect to the subspace topology of $D_\pi$ is called an
  \emph{open face} of $P$, and denoted $F_\pi$. Let $\Faces(P)$ and
  $\overline{\Faces}(P)$ denote the set of open and closed faces of
  $P$, respectively. Finally, define the faces of the graphic
  arrangement $\AAA(G)$ to be the faces of the posets over $G$:
  \[
 \Faces\AAA(G)=\bigcup_{\omega\in\Acyc(G)}\Faces(P(G,\omega))\,,\qquad\qquad
  \overline{\Faces}\,\AAA(G)
  =\bigcup_{\omega\in\Acyc(G)}\overline{\Faces}(P(G,\omega))\,.
  \]
  Faces of co-dimension $1$ are called \emph{facets}.
\end{defn}

\begin{rem}\label{rem:faces}
  The dimension of the face $F_\pi(P(G,\omega))$ is the number of strongly connected components of $\omega/\!\!\sim_\pi$.
  As long as $G$ is connected, there is a unique $1$-dimensional face of $\AAA(G)$, which is the
  line $x_1=\cdots =x_n$ and is contained in the closure of
  every chamber. There are no $0$-dimensional faces of
  $\AAA(G)$. The $n$-dimensional faces of $\AAA(G)$ are its chambers. Additionally, $\R^V$ is a disjoint union of open faces of
  $\AAA(G)$:
  \[
  \R^V=\dot{\bigcup_{F\in\Faces\AAA(G)}} F\,.
  \]
\end{rem}

If $P$ is a fixed poset over $G$, then there is a canonical
isomorphism between the lattice of closed face partitions and the
lattice of faces of $P$, given by the mapping $\pi\mapsto
F_\pi$. Recall that since $\pi$ is closed, $P_\pi=(\pi,\leq_P)$ is an acyclic preposet (i.e., poset) of size $|\pi|=r$. This induces an additional preposet over $V$ (i.e., of size $|V|=n$), which is $\leq_P$ with the additional relations that $x\sim_\pi y$ for
all $x,y\in B_i$. We will say that this is a \emph{preposet over
  $G$}, because it can be described by an (not necessarily acyclic)
orientation $\omega_\pi$ of $G$. The notation reflects the fact that this orientation can be constructed by starting with some $\omega\in\Acyc(G)$ and then making each edge bidirected if both endpoints are contained in the same block of $\pi$. Specifically, $\omega_\pi$ orients edge $\{i,j\}$ as $i\to j$ if $i\leq_P j$ and as $i\leftrightarrow j$ if additionally $i\sim_\pi
j$. Let $\Pre(G)$ be the set of all such orientations of $G$ that
arise in this manner. That is, 
\[
\Pre(G)=\{\omega_\pi\mid\omega\in\Acyc(G),\text{ $\pi$ closed face partition of $P(G,\omega)$}\}\,.
\]
When working with preposets over $G$, sometimes it is more convenient to quotient out by the strongly connected components and get an acyclic graph $\omega_\pi/\!\!\sim_\pi$. Note that this quotient is the same as $\omega'/\!\!\sim_\pi$ for at least one $\omega'\in\Acyc(G)$. In particular, $\omega'=\omega$ will always do. In summary, a preposet over $G$ can be expressed several ways:
\begin{enumerate}[(i)]
\item as a unique orientation $\omega_\pi$ of $G$, where $\pi$ is the partition into the strongly connected components;
\item as a unique acyclic quotient $\omega/\!\!\sim_\pi$ of an acyclic orientation $\omega\in\Acyc(G)$.
\end{enumerate}
Note that while the orientation $\omega_\pi$ and acyclic quotient $\omega/\!\!\sim_\pi$ are both unique to the preposet, the choice of representative $\omega$ is not. Regardless of how an element in $\Pre(G)$ is written, it induces a canonical partial order $(\pi,\leq)$ on the blocks of $\pi$. However, information is lost by writing it this way; in particular, the original graph $G$ cannot necessarily be determined from just $(\pi,\leq)$.

The mapping $\R^V\setmin\AAA(G)\stackrel{\alpha_G}{\longto}\Acyc(G)$
in Eq.~\eqref{eq:alpha-diagram1} can be extended to all of $\R^V$ by
adding both edges $i\to j$ and $j\to i$ if $x_i=x_j$. This induces a
bijection between the set $\Pre(G)$ of all preposets on $G$ and the
set of faces of the graphic hyperplane arrangement:
\begin{equation}
  \label{eq:alpha-diagram3}
  \xymatrix{\R^V \ar@{>>}[dr] \ar[rr]^{\alpha_G} &&
    \Pre(G) \\ & \Faces\AAA(G)\ar@{.>}[ur]_{c^{-1}} &\\ }
\end{equation}
Consequently, for any preposet $\omega_\pi$ over $G$, we can let $c(\omega_\pi)$ denote the open face of $\AAA(G)$ containing any (equivalently, all) $x\in\R^V$ such that $\alpha_G(x)=\omega_\pi$.

Moreover, if we restrict to the preposets on exactly $r$ strongly connected components, then the $\alpha_G$-fibers are the $r$-dimensional open faces of
$\AAA(G)$. If $x$ lies on a face $F_\pi(P)=F_{\bar\pi}(P)$ for some poset $P$ and closed face partition $\bar\pi=\cl_P(\pi)$, then the
preposet $P_\pi=(\pi,\leq_P)$ has vertex set $\pi=\{B_1,\dots B_r\}$; these are the strongly connected components of the orientation $\alpha_G(x)$.

%%=====================================================================
\section{Morphisms of ordinary posets}\label{sec:poset-morphisms}
%%=====================================================================

Poset isomorphisms are easy to describe both combinatorially and
geometrically. An \emph{isomorphism} between two finite posets
$P$ and $P'$ on vertex sets $V$ and $V'$ is a bijection
$\phi\colon V\to V'$ characterized
\begin{enumerate}
\item[$\bullet$] {\it combinatorially} by the condition that $i<_Pj$
  is equivalent to $\phi(i)<_{P'}\phi(j)$ for all $i,j\in V$;
\item[$\bullet$] {\it geometrically} by the equivalent condition that
  the induced isomorphism $\Phi$ on $\R^V\to\R^{V'}$ maps $c(P)$ to
  $c(P')$ bijectively.
\end{enumerate}
By ``induced isomorphism,'' we mean that $\Phi$ permutes the
coordinates of $\R^V$ in the same way that $\phi$ permutes the
vertices of $V$:
\begin{equation}\label{eq:induced}
  (x_1,x_2,\dots,x_n)\stackrel{\Phi}{\longmapsto}
  (x_{\phi^{-1}(1)},x_{\phi^{-1}(2)},\dots,x_{\phi^{-1}(n)})\,.
\end{equation}
Morphisms of ordinary posets are also well understood. The ``combinatorial'' definition is easiest to modify. If $P$ and $P'$ are as above, then a morphism, or \emph{order-preserving
  map}, is a function $\phi\colon V\to V'$ such that $i<_P j$ implies
$\phi(i)<_{P'}\phi(j)$ for all $i,j\in V$. The geometric
characterization is trickier because quotients, injections, and
extensions are inherently different. These three types of
order-preserving maps generate all poset morphisms, up to
isomorphism. Below we will review this and give a geometric
interpretation of each, which will motivate their toric analogues.

%%-------------------
\subsection{Quotient}

%%------------------------------------
\subsubsection{Contracting partitions}

Roughly speaking, a quotient morphism of a poset $P(G,\omega)$ is described combinatorially by contracting $\omega$ by the blocks of a partition
$\pi=\{B_1,\dots,B_r\}$ while preserving acyclicity. Geometrically, the chamber $c(P)$ is orthogonally projected to $\bar{F}_{\bar\pi}:=\overline{c(P)}\cap D_{\bar\pi}$, where $\bar\pi=\cl_P(\pi)$. This is the mapping
\begin{equation}\label{eq:d_pi}
    d_\pi\colon c(P)\longto D_\pi\,,\qquad
    d_\pi(x)=(x+D_{\bar\pi}^\perp)\cap D_{\bar\pi}\,.
\end{equation}
By construction, the image of this map is $F_{\bar\pi}$, which is a face of $P$ if $\pi$ is a face partition. Though the map $d_\pi$ extends to the closure $\overline{c(P)}$, it does \emph{not} do so in a well-defined manner; the image $d_\pi(x)$ for some $x$ on a hyperplane depends on the choice of $P$, and each hyperplane intersects two (closed) chambers along facets, and intersects the boundary of every chamber.

\begin{ex}\label{ex:K_3}
  Let $G=K_3$, the complete graph on $3$ vertices. There are six
  acyclic orientations of $G$, and three of them are shown in
  Figure~\ref{fig:K_3}. The curved arrows point to the chamber
  $c(P_i)$ of $\AAA(G)$ for each $P_i:=P(G,\omega_i)$, $i=1,2,3$. The
  intersection of each closed chamber $\overline{c(P_i)}$ with
  $[0,1]^3$ is the order polytope, $\O(P_i)$.

  Contracting $\omega_1$ and $\omega_2$ by the partition
  $\pi=\{B_1\!=\!\{1\},B_2\!=\!\{2,3\}\}$ yields a poset over
  $\{B_1,B_2\}$; these are shown directly below the orientations in
  Figure~\ref{fig:K_3}. Therefore, $\pi$ is closed with respect to
  $P_1$ and $P_2$. Geometrically, the flat $D_\pi$ intersects the
  closed chambers $\overline{c(P_i)}$ for $i=1,2$ in two-dimensional
  faces.

  In contrast, contracting $P_3$ by $\pi$ yields a preposet that is
  not a poset. Therefore, $\pi$ is not closed with respect to
  $P_3$. Indeed, $\bar\pi:=\cl_P(\pi)=123:=\{\{1,2,3\}\}$, and the flat $D_{\bar\pi}$
  intersects the closed chamber $\overline{c(P_3)}$ in a
  line. Modding out the preposet $P_\pi:=(\pi,\leq_{P_3})$ by its
  strongly connected components yields a one-element
  poset. Geometrically, the chamber $c(P_3)$ projects onto the
  one-dimensional face $\bar{F}_{\bar\pi}$.

To see why the map $d_\pi$ from Eq.~\eqref{eq:d_pi} does not extend to the closure of the chambers in a well-defined manner, consider the point $y$ shown in Figure~\ref{fig:K_3} that lies on the hyperplane $x_1=x_2$, and the same partition, $\pi=1/23$. The orthogonal projection $d_\pi\colon c(P_3)\to D_\pi$ as defined in Eq.~\eqref{eq:d_pi} and extended continuously to the closed chamber maps $\overline{c(P_3)}$ onto the line $x_1=x_2=x_3$. However, if $\overline{c(P_4)}$ is the other closed chamber containing $y$ (that is, the one for which $x_3\leq x_2\leq x_1$), then $d_\pi\colon c(P_4)\to D_\pi$ extended to the closure maps $\overline{c(P_4)}$ onto a 2-dimensional closed face $\bar{F}_\pi(P_4)=\overline{c(P_4)}\cap D_\pi$. The point $y$ is projected orthogonally onto the plane $D_\pi$, and does not end up on the line $x_1=x_2=x_3$.
\end{ex}

\begin{figure}\centering
  \tikzstyle{axis} = [draw, dashed,-stealth]
  \tikzstyle{to} = [draw,-stealth]
  \tikzstyle{v} = [circle, draw, fill=white,inner sep=0pt, 
    minimum size=1mm]  
  \begin{tikzpicture}[scale=0.6]
    \filldraw[thick,dashed,fill=lightgray,opacity=0.6] 
    (3.52,7.8)--(3.52,1.8)--(6,6)--cycle;
    \draw[dashed,opacity=0.7] {(4.76,3.9)--(3.52,7.8)};
    \filldraw[thick,dashed,fill=lightgray,opacity=0.6] 
    (6,6)--(9.52,7.8)--(3.52,1.8)--cycle;
    \draw[dashed,opacity=0.5] {(4.76,3.9)--(9.52,7.8)};
    \filldraw[thick,dashed,fill=lightgray,opacity=0.6] 
    (6,6)--(9.52,1.8)--(3.52,1.8)--cycle;
    \draw[dashed,opacity=0.5] {(4.76,3.9)--(9.52,1.8)};
    \filldraw[thick,dashed,fill=lightgray,opacity=0.6] 
    (0,6)--(6,6)--(3.52,1.8)--cycle;
    \draw[dashed,opacity=0.5] {(4.76,3.9)--(0,6)};
    \filldraw[thick,dashed,fill=lightgray,opacity=0.6] 
    (6,0)--(6,6)--(3.52,1.8)--cycle;
    \draw[dashed,opacity=0.5] {(4.76,3.9)--(6,0)};
    \filldraw[thick,dashed,fill=lightgray,opacity=0.6]%,pattern=north east lines] 
    (0,0)--(6,6)--(3.52,1.8)--cycle;
    \draw[dashed,opacity=0.5] {(4.76,3.9)--(0,0)};
    \draw {(0,0)--(0,6)--(6,6)--(6,0)--(0,0)};
    \draw {(3.52,7.8)--(9.52,7.8)--(9.52,1.8)};
    \draw {(0,6)--(3.52,7.8)}; 
    \draw {(6,0)--(9.52,1.8)}; 
    \draw {(6,6)--(9.52,7.8)};
    \draw (7.3,1.5) node{\scriptsize $x_3\!<\!x_1\!<\!x_2$};
    \draw (8.25,5) node{\scriptsize $x_1\!<\!x_3\!<\!x_2$};
    \draw (3.25,.5) node{\scriptsize $x_3\!<\!x_2\!<\!x_1$};
    \draw (6.25,7.25) node{\scriptsize $x_1\!<\!x_2\!<\!x_3$};
    \draw (1.3,2.75) node{\scriptsize $x_2\!<\!x_3\!<\!x_1$};
    \draw (2.2,6.3) node{\scriptsize $x_2\!<\!x_1\!<\!x_3$};
    \draw [axis] (9.52,1.8) to (10.27,1.8); \node at(10.2,2.1){\scriptsize $x_2$};
    \draw [axis] (3.52,7.8) to (3.52,8.5); \node at(3.1,8.2){\scriptsize $x_3$};
    \draw [axis] (0,0) to (-.7,-.35); \node at(-.5,.1){\scriptsize $x_1$};
    \node at (0,8) {$\R^3$};
    \node at (5.5,3) [label={[label distance=-.15cm]below:{\scriptsize $y$}}]{\scriptsize $\bullet$};
    \begin{scope}[shift={(11.5,7)}]
      \node (1a) at (0,0) {\scriptsize 1};
      \node (2a) at (1.73,1) {\scriptsize 2}; 
      \node (3a) at (1.73,-1) {\scriptsize 3}; 
      \draw[to] (1a) to (2a);
      \draw[to] (2a) to (3a);
      \draw[to] (1a) to (3a);
      \draw[dashed] (1.73,0) ellipse (.25 and 1.5);
      \draw[dashed] (0,0) circle (.25);
      \node (1A) at (0,-2) {\scriptsize $B_1$};
      \node (2A) at (1.73,-2) {\scriptsize $B_2$};
      \draw[to, shorten >= -3pt, shorten <= -3pt] (1A) to (2A);
      \path[to,out=160,in=45] (.25,1) to (-4.5,1);
      \node at (-5.5,1.7) {\scriptsize $\O(P(G,\omega_1))$};
      \node at (1,0) {\scriptsize $\omega_1$};
    \end{scope}
    \begin{scope}[shift={(11.5,2)}]
      \node (1c) at (0,0) {\scriptsize 1};
      \node (2c) at (1.73,1) {\scriptsize 2}; 
      \node (3c) at (1.73,-1) {\scriptsize 3}; 
      \draw[to] (3c) to (2c);
      \draw[to] (3c) to (1c);
      \draw[to] (1c) to (2c);
      \draw[dashed] (1.73,0) ellipse (.25 and 1.5);
      \draw[dashed] (0,0) circle (.25);
      \node (1C) at (0,-2) {\scriptsize $B_1$};
      \node (2C) at (1.73,-2) {\scriptsize $B_2$};
      \draw[to, shorten >= -3pt, shorten <= -3pt] (1C) to (2C);
      \draw[to, shorten >= -3pt, shorten <= -3pt] (2C) to (1C);
      \path[to,out=200,in=-30] (.25,-1) to (-3.5,-1.3);
      \node at (-2,-2.2) {\scriptsize $\O(P(G,\omega_3))$};
      \node at (1,0) {\scriptsize $\omega_3$};
    \end{scope}
    \begin{scope}[shift={(-3,2.75)}]
      \node (1b) at (0,0) {\scriptsize 1};
      \node (2b) at (1.73,1) {\scriptsize 2}; 
      \node (3b) at (1.73,-1) {\scriptsize 3}; 
      \draw[to] (2b) to (3b);
      \draw[to] (3b) to (1b);
      \draw[to] (2b) to (1b);
      \draw[dashed] (1.73,0) ellipse (.25 and 1.5);
      \draw[dashed] (0,0) circle (.25);
      \node (1B) at (0,-2) {\scriptsize $B_1$};
      \node (2B) at (1.73,-2) {\scriptsize $B_2$};
      \draw[to, shorten >= -3pt, shorten <= -3pt] (2B) to (1B);
      \path[to,out=80,in=140] (.75,1.5) to (2.7,2);
      \node at (.5,2.75) {\scriptsize $\O(P(G,\omega_2))$};
      \node at (1,0) {\scriptsize $\omega_2$};
    \end{scope}
  \end{tikzpicture}
  \caption{The hyperplane arrangement $\AAA(G)$ for $G=K_3$. Three
    orientations in $\Acyc(G)$ are shown, along with the corresponding
    chambers of $\AAA(G)$, and the preposet that results when
    contracting $P_i=P(G,\omega_i)$ by the partition
    $\pi=\{B_1\!=\!\{1\},B_2\!=\!\{2,3\}\}$ of $V$. The intersection
    of each (closed) chamber $\overline{c(P_i)}$ with $[0,1]^3$ is the
    order polytope $\O(P_i)$ of $P_i$. The point $y$ is supposed to lie on the hyperplane $x_1=x_2$.}
    \label{fig:K_3}
\end{figure}
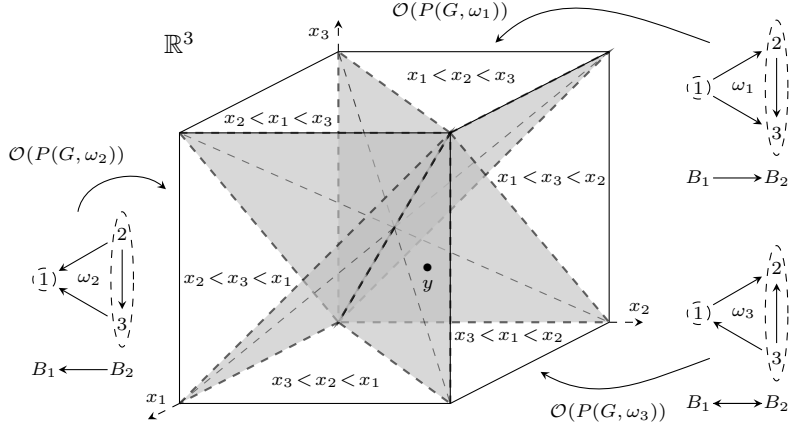

Despite this, there is a natural way to extend $d_\pi$ to all of $\R^V$, though not continuously. To do this, we first have to extend the notion of the closure of a partition $\pi$ with respect to a poset, to a preposet $P$ over $G$. This is easy, since the original definition did not specifically require $P$ to actually be a poset. Specifically, the closure of $\pi$ with respect to a preposet $P$ is the unique minimal coarsening $\bar\pi:=\cl_P(\pi)$ of $\pi$ such that $(\bar\pi,\leq_P)$ is acyclic. The map $d_\pi$ can now be extended to all of $\R^V$, as
\begin{equation}\label{eq:d_pi-extended}
d_\pi\colon\R^V\longto D_\pi\,,\qquad d_\pi(x)=(x+D_{\bar\pi}^\perp)\cap D_{\bar\pi}\,,\quad\text{where }\bar\pi=\overline{\pi(x)}:=\cl_{\alpha_G(x)}(\pi)\,,
\end{equation}
where $\alpha_G$ is the map from Eq.~\eqref{eq:alpha-diagram3} sending a point to the unique open face (i.e., preposet over $G$) containing it.

Let us return to the case where $P$ is a poset over $G$, and examine the case when $\pi$ is not a face partition of $P$. Indeed, for an arbitrary partition $\pi$ of $V$ with $\bar\pi=\cl_P(\pi)$, the subset $\bar{F}_\pi:=\overline{c(P)}\cap D_{\bar\pi}$ need not be a face of $P$; it could cut through the interior of the chamber. In this case, it is the face of at least one extension of $P$.  Specifically, let $G'_\pi$ be the graph formed by making each
block $B_i$ a clique, and let $G/\!\!\sim_\pi$ be the graph formed by
contracting these cliques into vertices, with loops and multiedges
removed. Clearly, $D_\pi$ is a flat of the graphic arrangement
$\AAA(G'_\pi)$ (this choice is not unique, but it is a canonical one that works). Thus, the set $\bar{F}_{\bar\pi}=\bar{F}_\pi=\overline{c(P)}\cap D_\pi$, for $\bar\pi=\cl_P(\pi)$, is a closed face of $\AAA(G'_\pi)$, and hence a face of some poset $P'$ over $G'_\pi$ for which $\bar\pi=\cl_{P'}(\pi)$. 

Whether or not $\pi$ is a face partition of a particular poset $P$ over $G$, the map $d_\pi$ in Eq.~\eqref{eq:d_pi-extended} projects a chamber $c(P)$ onto a flat $D_{\bar\pi}$ of $\AAA(G'_\pi)$, where $\bar\pi=\cl_P(\pi)$. From here, we need to project it homeomorphically onto a coordinate subspace of $\R^V$ so it is a chamber of a lower-dimensional arrangement. Specifically, for a  partition $\pi=\{B_1,\dots,B_r\}$, let $W\subseteq V$ be any subset
formed by removing all but $1$ coordinate from each $B_i$, and let
$\r_\pi\colon\R^V\longto\R^W$ be the induced projection. The
$r_\pi$-image of $\AAA(G)$ is the graphic arrangement of
$G/\!\!\sim_\pi$. The following ensures that $r_\pi$ is a
homeomorphism, and that the choice of $W$ does not matter. We omit the
elementary proof.

\begin{lem}\label{lem:homeomorphism}
  Let $\pi=\{B_1,\dots,B_r\}$ be a partition of $V$, and
  $W=\{b_1,\dots,b_r\}\subseteq V$ with $b_i\in B_i$. The restriction
  $\r_\pi|_{D_\pi}\colon D_\pi\longto\R^W$ is a homeomorphism.
  
  Moreover, all such projection maps for a fixed $\pi$ are
  topologically conjugate in the following sense: If
  $W'=\{b'_1,\dots,b'_r\}\subseteq V$ with $b'_i\in B_i$ and
  projection map $\r'_\pi|_{D_\pi}\colon D_\pi\longto\R^{W'}$, and
  $\sigma$ is the permutation of $V$ that transposes each $b_i$ with
  $b'_i$, then the following diagram commutes:
  \[
  \xymatrix{D_\pi\ar@{^{(}->}[r] & \R^V\ar[r]^{r_\pi}\ar[d]_\Sigma &
    \R^W\ar[d]^{\Sigma'} \\ D_\pi\ar@{^{(}->}[r] & \R^V\ar[r]^{r'_{\pi}} & \R^{W'}}
  \]
  Here, $\sigma'\colon W\to W'$ is the map $b_i\mapsto b'_i$, with
  $\Sigma$ and $\Sigma'$ being the induced linear maps as defined
  in Eq.~\eqref{eq:induced}.
\end{lem}

By convexity, $d_\pi$ induces a well-defined map
$\delta_\pi\colon\Faces\AAA(G)\longto\Faces\AAA(G'_\pi)$ making the
following diagram commute:
\begin{equation}\label{eq:delta_pi}
  \xymatrix{\R^V\ar[r]^{d_\pi}\ar[d] & D_\pi\ar[d]
    \\ \Faces\AAA(G)\ar@{.>}[r]^{\delta_\pi} & \Faces\AAA(G'_\pi)}
\end{equation}

The map $\delta_\pi$ is best understood by looking at a related map $\bar\delta_\pi$ on closed faces. Let $\bar{F}=\overline{c(P)}\cap D_\sigma$ be a closed face of $\AAA(G)$, for some closed face partition $\sigma\in\Pi_V$. Then the map $\bar\delta_\pi$ is defined by
\[
\bar\delta_\pi\colon\overline\Faces\,\AAA(G)\longto\overline\Faces\,\AAA(G'_\pi)
\,,\qquad\bar\delta_\pi\colon\overline{c(P)}\cap D_\sigma\longmapsto 
\overline{c(P)}\cap D_\sigma\cap D_\pi\,.
\]
The map $\delta_\pi$ is between the corresponding open faces. These faces are then mapped to faces of the arrangement $\AAA(G/\!\!\sim_\pi)$  under the projection $\r_\pi|_{D_\pi}\colon
D_\pi\longto\R^W$. [Alternatively, we could simply identity the quotient space $\R^V/D_\pi^\perp$ with $\R^W$.] 

To summarize, the open faces of $\AAA(G)$ arise from preposets $\omega=\omega_\sigma$ in $\Pre(G)$, where without loss of generality, the blocks for $\sigma\in\Pi_V$ are the strongly connected components. The contraction of this preposet formed by adding all relations (edges) of the form $(v,w)$ for $v,w\in B_i$ yields a preposet $\omega'_\pi$ over $G'_\pi$. Then, modding out by the strongly connected components yields an acyclic preposet, i.e., a poset. This two-step process is a composition of maps
\[
\Pre(G)\stackrel{q_\pi}{\longto}\Pre(G'_\pi)\stackrel{p_\pi}{\longto}\Pre(G/\!\!\sim_\pi)\,,\qquad\qquad
\omega_\sigma=\omega\stackrel{q_\pi}{\longmapsto}\omega'_\pi
\stackrel{p_\pi}{\longmapsto}\omega'_\pi/\!\!\sim_{\bar\pi}=\omega/\!\!\sim_{\bar\pi}.
\]
Here $\bar\pi=\cl_{P_\sigma}(\pi)$, the closure of $\pi$ with respect to the \emph{preposet} $P_\sigma$, which we have been denoting by $\omega_\sigma$ under a slight abuse of notation. 

Putting this all together gives a commutative diagram that illustrates the relationship between the points in $\R^V$, the open faces of the graphic arrangement $\AAA(G)$, and the preposets over $G$. The left column depicts the acyclic preposets -- those that are also posets.
\begin{equation*}\label{eq:6squares}
  \xymatrix{\R^V\setmin\AAA(G)\,\ar@{^{(}->}[r]\ar[d] &
    \R^V\ar[r]^{d_\pi}\ar[d] & D_\pi\ar[r]^{\r_\pi}\ar[d] & \R^W\ar[d]
    \\ \Chambers\AAA(G)\,\ar@{^{(}->}[r]\ar@{<->}[d] &
    \Faces\AAA(G)\ar[r]^{\delta_\pi}\ar@{<->}[d] &
    \Faces\AAA(G'_\pi)\ar[r]^{\rho_\pi}\ar@{<->}[d] &
    \Faces\AAA(G/\!\!\sim_\pi)\ar@{<->}[d] \\ \Acyc(G)\,\ar@{^{(}->}[r] &
    \Pre(G)\ar[r]^{\q_\pi} & \Pre(G'_\pi)\ar[r]^{\p_\pi} &
    \Pre(G/\!\!\sim_\pi) }
\end{equation*}

%%--------------------------------------
\subsubsection{Intervals and antichains}

Poset morphisms that are quotients are characterized geometrically by projecting the chamber $c(P)$ onto a flat $D_{\bar\pi}$ of $\AAA(G'_\pi)$ for some partition $\bar\pi=\cl_P(\pi)$, and then homeomorphically mapping this down to a chamber of a lower-dimensional graphic arrangement $\AAA(G/\!\!\sim_\pi)$. Equivalently,
contracting $P(G,\omega)$ by $\bar\pi$ yields an acyclic preposet
$P_{\bar\pi}=(\bar\pi,\leq_P)$. It is well known that contracting a poset by an interval or an antichain yields an acyclic preposet. Verification of this is elementary, but first recall how these are defined.

\begin{defn}
  Let $P$ be a poset over $V$. An \emph{interval} of $P$ is a subset
  $I\subseteq V$, sometimes denoted $[i,j]$, such that 
  \[
  I=\{x\in P:i\leq_Px\leq_P j\},\;\text{ for some fixed } i,j\in P\,.
  \]
  An \emph{antichain} of $P$ is a subset $A\subseteq V$ such that any
  two elements are incomparable.
\end{defn}

We will take a moment to understand how contracting an interval or antichain fits in the partition framework described above, which will help us understand the toric analogue. Given a nonempty subset $S\subseteq V$, define the partition $\pi_S$
of $V$ by
\begin{equation}\label{eq:pi_S}
  \pi_S=\{B_1=S,B_2,\dots,B_r\},\qquad\text{where $|B_i|=1$ for $i=2,\dots,r$}.
\end{equation}
Contracting an interval $I\subseteq V$ in a poset $P$ yields the poset
$P_{\pi_I}=(\pi_I,\leq_P)$. In this case, $\pi_I=\cl_P(\pi_I)$ is a face
partition and $F_\pi$ is a $(|V|-|I|+1)$-dimensional face of
$P$. Similarly, collapsing an antichain $A\subseteq V$ yields the
poset $P_{\pi_A}=(\pi_A,\leq_P)$. Note that $D_{\pi_I}$ is a flat of $\AAA(G)$ and lies on the boundary of $c(P)$, but the $(|V|-|A|+1$)-dimensional subspace $D_{\pi_A}$ cuts through the interior of $c(P)$. For both of these cases, $S=I$ and $S=A$, the subspace $D_{\pi_S}$ is trivially a flat of $\AAA(G'_{\pi_S})$.

%%------------------------------------------------------------------------
\subsection{Extension}\label{subsec:extensions}

Given two posets $P,P'$ on a set $V$, one says that {\it $P'$ is an
  extension of $P$} when $i<_Pj$ implies $i<_{P'}j$. In this case, the
identity map $P\longto P'$ is a poset morphism. Geometrically, $P'$ is
an extension of $P$ if and only if one has an inclusion of their open
polyhedral cones $c(P')\subseteq c(P)$. Each added relation $i<_{P'}j$
amounts to intersecting $c(P)$ with the half-space
$H_{i<j}:=\{x\in\R^V:x_i<x_j\}$.

%%-------------------------------------------------------------------------
\subsection{Inclusion} 

The last operation that yields a poset morphism is an injection
$\phi\colon P\into P'$. This induces a canonical inclusion
$\Phi\colon\R^V\into\R^{V'}$. Note that $i<_Pj$ implies $i<_{P'}j$,
but not necessarily vice-versa. Thus, up to isomorphism, an inclusion
can be decomposed into the composition $P\into P''\to P'$, where the
first map adds the elements $\{n+1,\dots,m\}$ to $P$ but no extra
relations, and then the map $P''\to P'$ is an extension. This gives an
inclusion of polyhedral cones:
\begin{align*}
  c(P'):=\{x\in\R^{V'}:x_i<x_j\text{ for }i<_{P'}j\}\subseteq
  c(P'')&=\{x\in\R^{V'}:x_i<x_j\text{ for } i<_Pj\} \\
  &\cong c(P)\times\R^{V'\setmin V}.
\end{align*}

%%------------------------------
\subsection{Summary}  

Up to isomorphism, every morphism of a poset $P=P(G,\omega)$ can be decomposed into a sequence of three steps:
\begin{enumerate}[(i)]
\item \emph{quotient}: Collapsing $G$ by a partition $\pi$ that preserves acyclicity of $\omega$ (projecting $c(P)$ to a flat $D_\pi$ of $\AAA(G'_\pi)$ for some closed partition $\pi=\cl_P(\pi)$).
\item \emph{inclusion}: Adding vertices (adding dimensions).
\item \emph{extension}: Adding relations (intersecting with
  half-spaces).
\end{enumerate}
In the special case of the morphism $P\longto P'$ being
surjective, the inclusion step is eliminated and the entire process
can be described geometrically by projecting $\overline{c(P)}$ to a
flat $D_\pi$ and then intersecting with a collection of half-spaces.

%%=====================================================================
\section{Toric posets and preposets}\label{sec:toric-posets}
%%=====================================================================

%%------------------------------------
\subsection{Toric chambers and posets}

Toric posets, introduced in~\cite{Develin:15}, arise from ordinary
(finite) posets defined by acyclic orientations under the equivalence
relation generated by converting maximal elements into minimal
elements, or sources into sinks.  Whereas an ordinary poset
corresponds to a chamber of a graphic arrangement $\AAA(G)$, a toric
poset corresponds to a chamber of a {\it toric graphic arrangement}
$\AAA_{\tor}(G)=\quot(\AAA(G))$, which is the image of $\AAA(G)$ under
the quotient map $\R^V \overset{\quot}{\longto} \R^V/\Z^V$. Elements
of $\AAA_{\tor}(G)$ are \emph{toric hyperplanes}
\[
H^{\tor}_{ij}:=\{x\in\R^V/\Z^V:\modone{x_i}=\modone{x_j}\}=\quot(H_{ij})\,.
\]
\begin{defn}
  \label{toric-poset-defn}
  A connected component $c$ of the complement $\R^V/\Z^V \setmin
  \AAA_{\tor}(G)$ is called a {\it toric chamber} for $G$, or simply a chamber of $\AAA_{\tor}(G)$.  Let
  $\Chambers \AAA_{\tor}(G)$ denote the set of all chambers of
  $\AAA_{\tor}(G)$.
  
  A {\it toric poset} $P$ is a set $c$ that arises as a toric chamber
  for at least one graph $G$. We may write $P=P(c)$ or $c=c(P)$,
  depending upon the context.
\end{defn}
If we fix a graph $G=(V,E)$ and consider the arrangement
$\AAA_{\tor}(G)$, then each point in $\R^V/\Z^V$ naturally determines
a preposet on $G$ via a map $\bar\alpha_G\colon
x\mapsto\omega(x)$. Explicitly, for
$x=(\modone{x_1},\ldots,\modone{x_n})$ in $\R^V/\Z^V$, the directed
graph $\omega(x)$ is constructed by doing the following for each edge
$\{i,j\}$ in $E$:
\begin{itemize}
\item If $\modone{x_i}\leq\modone{x_j}$, then include edge $i\to j$;
\item If $\modone{x_j}\leq\modone{x_i}$, then include edge $j\to i$.
\end{itemize}
The mapping $\bar\alpha_G$ is essentially the same as $\alpha_G$ from
Eq.~\eqref{eq:alpha-diagram3} except done modulo $1$, so many of its
properties are predictably analogous. For example, the undirected
version of $\omega(x)$ is $G$. The edge $\{i,j\}$ is bidirected in
$\omega(x)$ if and only if $\modone{x_i}=\modone{x_j}$. Therefore,
$\omega(x)$ is acyclic if and only if $x$ lies in
$\R^V/\Z^V\setmin\AAA_{\tor}(G)$; in this case $\omega(x)$ describes a
poset. Otherwise it describes a preposet (that is not a
poset). Modding out by the strongly connected components yields an acyclic graph $\omega(x)/\!\!\sim_x$ that describes a poset.

\begin{defn}
  When two preposets $\omega(x)$ and $\omega(y)$ are such that the directed graphs $\omega(x)/\!\!\sim_x$ and $\omega(y)/\!\!\sim_y$ differ only by converting a source vertex (equivalence class) into a sink, or vice-versa, we say they differ by a \emph{flip}. The transitive closure of the flip operation generates an equivalence relation on $\Pre(G)$, denoted by $\equiv$.
\end{defn}

In the special case of restricting to preposets that are acyclic, we get $\Acyc(G)\subseteq\Pre(G)$ and a bijective correspondence between toric posets and chambers of toric graphic arrangements. This is Theorem 1.4 in~\cite{Develin:15}. A generalization of this to a bijection between toric preposets and faces of the toric graphic arrangement appears later in this section (Proposition~\ref{prop:toric-bijection}).

\begin{thm}{(\cite{Develin:15}, Theorem 1.4)}
\label{thm:bijection}
The map $\bar{\alpha}_G$ induces a bijection
between $\Chambers \AAA_{\tor}(G)$ and $\Acyc(G)/\!\!\equiv$ as follows:

\begin{equation}\label{eq:toric-chamber-bijection}
\xymatrix{\R^V/\Z^V \setmin \AAA_{\tor}(G) \ar@{>>}[d]
  \ar[r]^{\bar{\alpha}_G} & \Acyc(G) \ar@{>>}[d] \\ \Chambers
  \AAA_{\tor}(G) \ar@{.>}[r]_{\bar{\alpha}_G} & \Acyc(G)/\!\!\equiv}
\end{equation}
In other words, two points $x,x'$ in $\R^V/\Z^V \setmin \AAA_{\tor}(G)$
have $\bar{\alpha}_G(x)\equiv\bar{\alpha}_G(x')$ if and only if
$x,x'$ lie in the same toric chamber of $\AAA_{\tor}(G)$.
\end{thm}

By Theorem~\ref{thm:bijection}, every pair $(G,[\omega])$ of a graph
$G$ and $\omega\in\Acyc(G)$ determines a toric poset, and we denote
this by $P(G,[\omega])$. \emph{Specifically, $P(G,[\omega])$ is the
  toric poset $P(c)$ such that $\bar\alpha_G(c)=[\omega]$}.
If the graph $G$ is understood, then we may
denote the corresponding toric chamber by
$c_{[\omega]}:=c(P(G,[\omega]))$.

If $G=(V,E)$ is fixed, then the unit cube $[0,1]^V$ in $\R^V$ is the
union of order polytopes $\O(P(G,\omega))$, any two of which only
intersect in a subset of a flat of $\AAA(G)$:
\begin{equation}\label{eq:unit-cube}
  [0,1]^V=\bigcup_{\omega\in\Acyc(G)}\O(P(G,\omega))\,.
\end{equation}
When $G$ is understood, we will say that the order polytopes $\O(P(G,\omega))$ and $\O(P(G,\omega'))$ are torically equivalent whenever $\omega'\in[\omega]$. 
Under the natural quotient $q\colon[0,1]^V\to\R^V/\Z^V$, each order
polytope $\O(P(G,\omega))$ is mapped into the closed toric chamber
$\bar{c}_{[\omega]}$. Moreover, by Theorem~\ref{thm:bijection}, the closed chambers
of $\AAA_{\tor}(G)$ are unions of $q$-images of torically equivalent order polytopes. 

\begin{cor}\label{cor:order-polytope}
  Let $P=P(G,[\omega])$ be a toric poset, and
  $\quot\colon[0,1]^V\to\R^V/\Z^V$ the natural quotient. The closure
  of the chamber $c(P)$ is
  \[
  \overline{c(P)}=\bigcup_{\omega'\in[\omega]}q(\O(P(G,\omega')))\,.
  \]
\end{cor}

%%--------------------------------------------------------------
\subsection{Toric faces and preposets}\label{subsec:toric-faces}

Let $P=P(c)$ be a toric poset over $G=(V,E)$. To define objects like a
face of $P$ or its dimension, it helps to first lift $c$ up to a
chamber of the \emph{affine graphic arrangement} which lies in $\R^V$:
\begin{equation}
  \label{affine-arrangement-defn}
  \AAA_{\aff}(G):= \quot^{-1}(\AAA_{\tor}(G)) = \quot^{-1}(\quot(\AAA(G)))
  =\bigcup_{\substack{ \{i,j\} \in E\\ k \in \Z}} \{x \in \R^V: x_i = x_j+k \}.
\end{equation}
The affine chambers are open unbounded convex polyhedral regions in $\R^V$, the universal cover of $\R^V/\Z^V$. The path lifting property guarantees that two points $x$ and $y$ in $\R^V/\Z^V\setmin\AAA_{\tor}(G)$ are in the same toric chamber if and only if they have lifts $\lift{x}$ and $\lift{y}$ that are in the same affine chamber. Moreover, since Corollary~\ref{cor:order-polytope} characterizes the closed toric chamber $\overline{c(P)}$ as a union of torically equivalent order polytopes under a universal covering map, each closed affine chamber is a union of translated copies of torically equivalent order polytopes in $\R^V$. 

We usually denote an affine chamber by $\lift{c}$ or $c^{\aff}$. Each
hyperplane $H^{\tor}_{ij}$ has a unique preimage containing the
origin in $\R^V$ called its \emph{central preimage}; this is the
ordinary hyperplane $H_{ij}=\{x\in\R^V:x_i=x_j\}$. Thus, the set of
central preimages of $\AAA_{\tor}(G)$ is precisely the graphic
arrangement $\AAA(G)$ in $\R^V$. Each closed affine chamber $\bar{c}^{\aff}$
contains at most one order polytope $\O(P(G,\omega))$ for
$\omega\in\Acyc(G)$. Affine chambers whose closures contain precisely one order polytope $\O(P(G,\omega))$ are \emph{central affine chambers}.

We will call nonempty sets that arise as intersections of hyperplanes
in $\AAA_{\aff}(G)$ \emph{affine flats} and nonempty sets that are
intersections of hyperplanes in $\AAA_{\tor}(G)$ \emph{toric
  flats}. Since the toric flats have a nonempty intersection, they
form a lattice that is denoted $L(\AAA_{\tor}(G))$, and partially ordered by
reverse inclusion.

Since a toric flat of $\AAA_{\tor}(G)$ is the image of a flat of
$\AAA(G)$, it too is determined by a partition $\pi$ of $V$, and so
it is of the form
\begin{equation}\label{eq:toric-D(S)}
  D^{\tor}_\pi=\{x\in\R^V/\Z^V:\text{ $x_i=x_j$ for every pair $i,j$ in the same
    block $B_k$ of $\pi$}\}=q(D_\pi)\,.
\end{equation}
Since $\R^V\stackrel{\quot}{\longto}\R^V/\Z^V$ is a covering map, it
is well-founded to declare the dimension of a toric flat
$D^{\tor}_\pi$ in $\R^V/\Z^V$ to be the same as the dimension of its
central preimage $D_\pi$ in $\R^V$.

Recall that a partition $\pi=\{B_1,\dots,B_r\}$ is compatible with an ordinary poset
$P$ if contracting the blocks of $\pi$ yields a preposet
$P_\pi=(\pi,\leq_P)$ that is acyclic (also a poset). The notion of
compatible partitions does not carry over well to toric posets,
because compatibility is not preserved by toric
equivalence. Figure~\ref{fig:K_3} shows an example of this: the
preposets $(\pi,\leq_{P_1})$ and $(\pi,\leq_{P_2})$ are
acyclic but $(\pi,\leq_{P_3})$ is not. Despite this, every
set $D^{\tor}_\pi$, whether or not it is a toric flat of $\AAA_{\tor}(G)$, intersects the closed
toric chamber $\overline{c(P)}$ in at least the line $x_1=\cdots=x_n$. We denote this intersection by
\begin{equation}\label{eq:F^tor_pi}
  \bar{F}^{\tor}_\pi(P):=\overline{c(P)}\cap D^{\tor}_\pi\,.
\end{equation}
If $D^{\tor}_\pi$ does not intersect $c(P)$, then we say that $\pi$ is a \emph{toric face partition}, since it intersects the closed toric chamber along its boundary. Compare this to the definition of face partitions of an ordinary poset $P(G,\omega)$, which are those $\pi\in\Pi_V$ characterized by $D_\pi$ being a flat of the graphic arrangement of the transitive closure, or equivalently, by $D_\pi\cap c(P(G,\omega))=\varnothing$. The transitive closure $\bar{G}(P(G,\omega))$ is formed from $G$ by adding all additional edges $\{i,j\}$ such that $H_{i,j}\cap c(P(G,\omega))=\varnothing$. Similarly, we can define the \emph{toric transitive closure} of $P(G,[\omega])$ as the graph $G$ along with the extra edges $\{i,j\}$ such that $H_{i,j}^{\tor}\cap c(P)=\varnothing$. This was done in \cite{Develin:15}, and we will return to it in the Section~\ref{subsec:toric-hasse} when we discuss toric Hasse diagrams.

Now, let $\pi\in\Pi_V$ be an arbitrary partition. Since flats of $\AAA_{\tor}(K_V)$ are closed under intersections, there is a unique maximal toric subspace $D^{\tor}_{\bar\pi}$ (that is, of minimal dimension) for which
$\bar{F}^{\tor}_\pi=\overline{c(P)}\cap D^{\tor}_{\bar\pi}$. The
partition $\bar\pi$ is the unique minimal coarsening of $\pi$ for which
$\bar{F}^{\tor}_\pi=\bar{F}^{\tor}_{\bar\pi}$, and it is the
lattice-join of all such partitions. We call it the \emph{closure} of
$\pi$ with respect to the toric poset $P$, denoted $\cl_P^{\tor}(\pi)$, and we define $\dim(\bar{F}^{\tor}_\pi):=\dim(D^{\tor}_{\bar\pi})$. A partition $\pi$ is \emph{closed} with respect to the toric poset $P$ if $\pi=\cl_P^{\tor}(\pi)$. Note that the closure is defined for all partitions, not just toric face partitions.

\begin{defn}
  A set $F\subseteq\R^V/\Z^V$ is a \emph{closed face}
  of the toric poset $P$ if $F=\bar{F}^{\tor}_\pi=\overline{c(P)}\cap
  D_\pi^{\tor}$ for some closed toric face partition $\pi=\cl_P^{\tor}(\pi)$. The
  interior of $\bar{F}^{\tor}_\pi$ with respect to the subspace
  topology of $D^{\tor}_\pi$ is called an \emph{open face} of $P$, and
  denoted $F^{\tor}_\pi$. Let $\Faces(P)$ and $\overline{\Faces}(P)$
  denote the set of open and closed faces of $P$, respectively.
  Finally, define the faces of the toric graphic
  arrangement $\AAA_{\tor}(G)$ to be the faces of the toric posets over $G$:
  \[
  \Faces\AAA_{\tor}(G)
  =\bigcup_{\omega\in\Acyc(G)}\Faces(P(G,[\omega]))\,,\qquad\qquad
  \overline{\Faces}\,\AAA_{\tor}(G)
  =\bigcup_{\omega\in\Acyc(G)}\overline{\Faces}(P(G,[\omega]))\,.
  \]
  Toric faces of co-dimension $1$ are called \emph{facets}.
\end{defn}

The following remark is the toric analogue of Remark~\ref{rem:faces}.

\begin{rem}\label{rem:toric-faces}
Let $P=P(G,[\omega])$ be a toric poset. 
The dimension of $\bar{F}_\pi^{\tor}(P):=\overline{c(P)}\cap D_\pi^{\tor}$ is simply the maximum dimension of $\overline{c^{\aff}(P)}\cap D_\pi$ taken over all affine chambers that descend down to $\overline{c(P)}$. Since closed affine chambers are unions of translations of order polytopes, this is the maximum dimension of $\overline{c(P(G,\omega'))}\cap D_\pi$ taken over all $\omega'\in[\omega]$. In other words, 
\[
\dim F_\pi^{\tor}(P(G,[\omega]))=\max_{\omega'\in[\omega]}\dim F_\pi(P(G,\omega'))\,.
\]
On the level of graphs, this is the maximum number of strongly connected components that $\omega'/\!\!\sim_\pi$ can have for some $\omega'\in[\omega]$. In particular, a partition $\pi$ is closed with respect to $P(G,[\omega])$ if and only if $\omega'/\!\!\sim_\pi$ is acyclic for \emph{some} $\omega'\in[\omega]$.

  As long as $G$ is connected, there is a unique $1$-dimensional face of $\AAA_{\tor}(G)$, which is the
  line $x_1=\cdots =x_n$ and is contained in the closure of
  every chamber. There are no $0$-dimensional faces of
  $\AAA_{\tor}(G)$. The $n$-dimensional faces of $\AAA_{\tor}(G)$ are its chambers. Additionally, $\R^V/\Z^V$ is a disjoint union of open faces of
  $\AAA_{\tor}(G)$:
  \begin{equation}\label{eq:toric-faces}
  \R^V/\Z^V=\dot{\bigcup_{F\in\Faces\AAA_{\tor}(G)}} F\,.
  \end{equation}
\end{rem}

As in the case of ordinary posets, there is a canonical bijection between the closed toric face partitions of $P$ and open faces (or closed faces) of $P$, via $\pi\mapsto F_\pi^{\tor}$. To classify the faces of a toric poset, it suffices to classify the closed toric face partitions. 

\begin{thm}\label{thm:toric-face-partitions}
  Let $P=P(G,[\omega])$ be a toric poset over $G=(V,E)$. A partition $\pi$ of $V$ is a closed toric face partition of $P$ if and only if it is connected and compatible with $P(G,\omega')$, for some $\omega'\in[\omega]$.
\end{thm}

The proof of Theorem~\ref{thm:toric-face-partitions} will be done later in this section, after the following lemma, which establishes that $\cl_P^{\tor}$ is a closure operator \cite{ward1942closure} on the partition lattice $\Pi_V$ and compares it with $\cl_P$.

\begin{lem}\label{lem:partitions}
Let $\omega$ be an acyclic orientation of a graph $G=(V,E)$, and $\pi$ a partition of $V$. 
\begin{enumerate}[(a)]
\item If $\pi$ is closed with respect to $P(G,\omega)$, then $\pi$ is closed with respect to $P(G,[\omega])$.
\item Closure is monotone: if $\pi\leq_V\pi'$, then $\cl_P^{\tor}(\pi)\leq_V\cl_P^{\tor}(\pi')$.
\item If $\pi\leq_V\pi'\leq_V\cl_P^{\tor}(\pi)$, then $\cl_P^{\tor}(\pi')=\cl_P^{\tor}(\pi)$.
\item $\cl_{P(G,[\omega])}^{\tor}(\pi)\leq_V\cl_{P(G,\omega)}(\pi)$.
\end{enumerate}
\end{lem}

\begin{proof}
If $\pi$ is closed with respect to $P(G,\omega)$, then the preposet $\omega/\!\!\sim_\pi$ is acyclic. By Remark~\ref{rem:toric-faces}, this means that $\pi$ is closed with respect to $P(G,[\omega])$, which establishes (a).

Part~(b) is obvious. Part~(c) follows from taking the closure of each term in the chain of inequalities $\pi\leq_V\pi'\leq_V\cl_P^{\tor}(\pi)$:
\[
\cl_P^{\tor}(\pi)\leq_V\cl_P^{\tor}(\pi')\leq_V
\cl_P^{\tor}\!\left(\cl_P^{\tor}(\pi)\right)=\cl_P^{\tor}(\pi)\,.
\]
To prove (d), let $\bar\pi=\cl_{P(G,\omega)}(\pi)$, which is closed with respect to $P(G,\omega)$. By (a), $\bar\pi$ is closed with respect to $P(G,[\omega])$. Using this, along with (b) applied to $\pi\leq_V\bar\pi$, yields
\[
\cl_{P(G,[\omega])}^{\tor}(\pi)\leq_V
\cl_{P(G,[\omega])}^{\tor}(\bar\pi)=\bar\pi=\cl_{P(G,\omega)}(\pi)\,,
\]
whence the theorem.
\end{proof}

\begin{ex}
For an example both of where the converse to Lemma~\ref{lem:partitions}(a) fails, and where $\cl_{P(G,[\omega])}^{\tor}(\pi)\lneq_V\cl_{P(G,\omega)}(\pi)$, consider $G=K_3$, and the partition $\pi=1/23$. Using the same notation as in Example~\ref{ex:K_3} and Figure~\ref{fig:K_3}, we see that $\cl_{P(G,\omega_3)}(\pi)=123$, because the intersection of $\overline{c(P(G,\omega_3))}$ and $D_\pi$ is one-dimensional. Equivalently, the preposet $\omega_3/\!\!\sim_\pi$ has one strongly connected component. In contrast, the intersection of $\overline{c(P(G,[\omega_3])}$ with $D_\pi^{\tor}$ is two-dimensional. Indeed, the preposets $\omega_1/\!\!\sim_\pi$ and $\omega_2/\!\!\sim_\pi$ both have two strongly connected components, and $\omega_1\equiv\omega_2\equiv\omega_3$. Therefore,
\[
1/23=\pi=\cl_{P(G,[\omega_3])}^{\tor}(\pi)\lneq_V \cl_{P(G,\omega_3)}(\pi)=123\,,
\]
and so $\pi$ is closed with respect to $P(G,[\omega_3])$ but not with respect to $P(G,\omega_3)$.
\end{ex}

\begin{proof}[Proof of Theorem~\ref{thm:toric-face-partitions}]
Suppose $\pi$ is a closed toric face partition of $P=P(G,[\omega])$, and that $\bar{F}_\pi^{\tor}:=\overline{c(P)}\cap D_\pi^{\tor}$ has dimension $k$. Then for some $\omega'\in[\omega]$, the hyperplane $D_\pi$ must intersect the order polytope $\O(P(G,\omega'))$ in a $k$-dimensional face. Thus, $\bar{F}_\pi:=\overline{c(P(G,\omega'))}\cap D_\pi$ has dimension $k$ for some $\omega'\in[\omega]$, and so $\pi$ is a face partition of $P(G,\omega')$. To see why $\pi$ is closed with respect to $P(G,\omega')$, suppose there were a coarsening $\pi'\geq_V\pi$ such that 
\begin{equation}\label{eq:pi'}
\bar{F}_{\pi'}=\overline{c(P(G,\omega'))}\cap D_{\pi'}=\overline{c(P(G,\omega'))}\cap D_\pi=\bar{F}_\pi\,.
\end{equation}
It suffices to show that $\pi'=\pi$. Descending down to the torus, the intersection $\bar{F}_{\pi'}^{\tor}:=\overline{c(P)}\cap D_{\pi'}^{\tor}$ must have dimension at least $k$ by Eq.~\eqref{eq:pi'}, but no more than $k$ because $\bar{F}_{\pi'}^{\tor}\subseteq\bar{F}_\pi^{\tor}$. Thus, we have equality $\bar{F}_{\pi'}^{\tor}=\bar{F}_\pi^{\tor}$, which means $\pi'\leq_V\cl_P^{\tor}(\pi)=\pi$, whence $\pi'=\pi$. Since $\pi$ is a closed face partition with respect to $P(G,\omega')$, it is connected and compatible with $P(G,\omega')$ by Theorem~\ref{thm:face-partitions}.

Conversely, suppose that $\pi$ is connected and compatible with respect to $P(G,\omega')$ for some $\omega'\in[\omega]$. By Theorem~\ref{thm:face-partitions}, $\pi$ is a closed face partition of $P(G,\omega')$. Since $\pi$ is connected, $D_\pi$ is a flat of $\AAA(G)$. Therefore, $D_\pi^{\tor}$ is a toric flat of $\AAA_{\tor}(G)$, and so $\bar{F}_\pi^{\tor}:=\overline{c(P)}\cap D_\pi^{\tor}$ is a face of the toric poset $P(G,[\omega'])=P(G,[\omega])=P$. Therefore, $\pi$ is a toric face partition. Closure of $\pi$ with respect to $P=P(G,[\omega])=P(G,[\omega'])$ follows immediately from Lemma~\ref{lem:partitions}(a) applied to the fact that $\pi$ is closed with respect to $P(G,\omega')$. 
\end{proof}

Unlike the ordinary case, where faces of posets are literally faces of a convex polyhedral cone, it is not quite so ``geometrically obvious'' what subsets can be toric faces. The following example illustrates this.

\begin{ex}\label{ex:V=2}
  There are only two simple graphs $G=(V,E)$ over $V=\{1,2\}$: The
  edgeless graph $G_0=(V,\varnothing)$, and the complete graph
  $K_2=(V,\{\{1,2\}\})$. For both graphs, the complement $\R^2/\Z^2\setmin\AAA(G)$ is
  connected. The respective chambers are
  \[
  c_0:=\R^2/\Z^2\,,\qquad \text{and}\qquad c:=\R^2/\Z^2\setmin H_{12}^{\tor}\,,
  \]
  and so they represent different toric posets, $P(c_0)$ and $P(c)$.
  Despite this, these chambers have the same topological closures:
  $\bar{c_0}=\bar{c}=\R^2/\Z^2$. The lattice of flats of $G_0$ contains one   element: $\AAA(G_0)=\{D^{\tor}_{1/2}\}$, and this flat arises from the permutation $\pi=1/2$.  
  The lattice of flats of $K_2$ has two elements:
  $\AAA(K_2)=\{D^{\tor}_{1/2},D^{\tor}_{12}\}$, where
  $D^{\tor}_{1/2}=\R^V/\Z^V$, and $D^{\tor}_{12}=H_{12}^{\tor}=\{x_1=x_2\}$. Thus,
  the closed faces of the corresponding toric posets are
  \[
  \overline{\Faces}(P(c_0))=\{\R^V/\Z^V\}\,,\qquad
  \text{and}\qquad\overline{\Faces}(P(c))
  =\{\R^V/\Z^V,\;H_{12}^{\tor}\}.
  \]
\end{ex}

The subtlety in Example~\ref{ex:V=2} does not arise for ordinary
posets, because distinct ordinary posets never have chambers with the
same topological closure. In contrast, if $G=(V,E)$ and $G'=(V,E')$
are both forests, then $\AAA_{\tor}(G)$ and $\AAA_{\tor}(G')$ both
have a single toric chamber. This is because the number of chambers is counted by the Tutte polynomial evaluation $T_G(1,0)$, which is always $1$ for a forest; see~\cite{Develin:15}. In this case, the closures of both chambers will be all of $\R^V/\Z^V$. A more complicated example involving a toric poset over a graph with three vertices, will appear soon in Example~\ref{ex:K_3-part2}.

Recall the map $\bar{\alpha}_G$ from Eq.~\eqref{eq:toric-chamber-bijection} that
sends a point $x$ in $\R^V/\Z^V$ to a preposet $\omega(x)$. By
Theorem~\ref{thm:bijection}, when restricted to the points in
$\R^V/\Z^V\setmin\AAA_{\tor}(G)$, this map induces a bijection between
toric posets and toric chambers. Toric faces $F_\pi^{\tor}$ that are open in $D_\pi^{\tor}$ are chambers in lower-dimensional arrangements that are contractions of $\AAA_{\tor}(G)$, namely by the subspace $D_\pi^{\tor}$. Thus, the bijection between toric equivalence classes of $\Acyc(G)$ ($n$-element preposets) and toric chambers ($n$-dimensional faces) extends naturally to a bijection between toric preposets over $G$ and open faces of $\AAA_{\tor}(G)$.

\begin{prop}\label{prop:toric-bijection}
  The map $\bar{\alpha}_G$ induces a bijection between
  $\Faces\AAA_{\tor}(G)$ and $\Pre(G)/\!\!\equiv$ as follows:
  \begin{equation}\label{eq:toric-face-bijection}
    \xymatrix{\R^V/\Z^V \ar@{>>}[d]
      \ar[r]^{\bar{\alpha}_G} & \Pre(G) \ar@{>>}[d] \\ \Faces
      \AAA_{\tor}(G) \ar@{.>}[r]_{\bar{\alpha}_G} & \Pre(G)/\!\!\equiv}
  \end{equation}
  In other words, given two points $x,x'$ in $\R^V/\Z^V$ the equivalence $\bar{\alpha}_G(x)\equiv\bar{\alpha}_G(x')$ holds if and only if
  $x,x'$ lie on the same open face of $\AAA_{\tor}(G)$.
\end{prop}

\begin{defn}\label{defn:toric-preposet}
  A \emph{toric preposet} is a set that arises as an open face of a toric
  poset $P=P(c)$ for at least one graph $G$.
\end{defn}

If $x$ lies on a toric face $F^{\tor}_\pi$ of $P$, where (without loss of generality) $\pi=\cl_P^{\tor}(\pi)$,
then the strongly connected components of the preposet $\bar\alpha_G(x)$ are $\pi=\{B_1,\dots B_r\}$. 

\begin{ex}\label{ex:K_3-part2}
  Let $G=K_3$, as in Example~\ref{ex:K_3}. The six acyclic
  orientations of $G$ fall into two toric equivalence classes. The
  three orientations shown in Figure~\ref{ex:K_3} comprise one
  class, and so the corresponding toric poset is $P=P(G,[\omega_i])$
  for any $i=1,2,3$. Equivalently, the closed toric chamber is a union of 
  order polytopes under the natural quotient map: 
  \[
  \overline{c(P)}=\bar{c}_{[\omega_i]}=\bigcup_{i=1}^3q(\O(P(G,\omega_i)))\,.
  \]
  This should be visually clear from Figure~\ref{ex:K_3}.
  The two chambers in $\AAA_{\tor}(G)$ are the three-dimensional faces
  of $P$. Each of the three toric hyperplanes in 
  $\AAA_{\tor}(G)=\{H_{12}^{\tor},H_{13}^{\tor},H_{23}^{\tor}\}$ are
  two-dimensional faces of $P$, and these (toric preposets) correspond
  to the following toric equivalence classes of size-2 preposets $\omega_\pi/\!\!\sim_\pi$ over
  $K_3$:
  \begin{center}
    \tikzstyle{to} = [draw,-stealth]
    \begin{tikzpicture}[shorten >= -2.5pt, shorten <= -2.5pt]
      \begin{scope}[shift={(10,0)}]
        \node (1a) at (0,0) {\scriptsize $\{1\}$};
        \node (23a) at (2,0) {\scriptsize $\{2,3\}\;,$};
        \draw[to] (1a) to (23a);
        \node (1b) at (0,-.75) {\scriptsize $\{1\}$};
        \node (23b) at (2,-.75) {\scriptsize $\{2,3\}$};
        \draw[to] (23b) to (1b);
        \node at (-.5,0) {$\Big\{$};
        \node at (2.5,-.75) {$\Big\}$};
      \end{scope}
      \begin{scope}[shift={(5,0)}]
        \node (2a) at (0,0) {\scriptsize $\{2\}$};
        \node (13a) at (2,0) {\scriptsize $\{1,3\}\;,$};
        \draw[to] (2a) to (13a);
        \node (2b) at (0,-.75) {\scriptsize $\{2\}$};
        \node (13b) at (2,-.75) {\scriptsize $\{1,3\}$};
        \draw[to] (13b) to (2b);
        \node at (-.5,0) {$\Big\{$};
        \node at (2.5,-.75) {$\Big\}$};
      \end{scope}
      \begin{scope}[shift={(0,0)}]
        \node (3a) at (0,0) {\scriptsize $\{3\}$};
        \node (12a) at (2,0) {\scriptsize $\{1,2\}\;,$};
        \draw[to] (3a) to (12a);
        \node (3b) at (0,-.75) {\scriptsize $\{3\}$};
        \node (12b) at (2,-.75) {\scriptsize $\{1,2\}$};
        \draw[to] (12b) to (3b);
        \node at (-.5,0) {$\Big\{$};
        \node at (2.5,-.75) {$\Big\}$};
      \end{scope}
    \end{tikzpicture}
  \end{center}
  The toric flat $D_{\{V\}}$ is the unique one-dimensional face of
  $P$, and this corresponds to the unique size-$1$ preposet over $K_3$; when
  $x_1=x_2=x_3$, which is trivially in its own toric equivalence
  class.
\end{ex}

%%=====================================================================
\section{Toric intervals and antichains}\label{sec:toric-intervals}
%%=====================================================================

Collapsing an interval or antichain of an ordinary poset defines a poset morphism. This remains true in the toric case, as will be shown in Section~\ref{sec:toric-morphisms}, though the toric analogues of these concepts are trickier to define. Toric antichains were introduced in~\cite{Develin:15}, but toric intervals are new to this paper. First, we need to review some terminology and results about toric total orders, chains, transitivity, and Hasse diagrams. This will also be needed to study toric order ideals and filters in Section~\ref{sec:toric-ideals}. Much of the content in Sections~\ref{subsec:toric-total-orders}--\ref{subsec:toric-hasse} can be found in~\cite{Develin:15}. Throughout, $G=(V,E)$ is a fixed undirected graph with $|V|=n$, and coordinates $x_i$ of points $x=(x_1,\dots,x_n)$ in a toric chamber $c(P)$ are assumed to be reduced modulo $1$, i.e., $x_i\in[0,1)$. 

%%---------------------------------------------------------------------------
\subsection{Toric total orders}\label{subsec:toric-total-orders}

A toric poset $P'$ is a \emph{total toric order} if $c(P')$ is a chamber of $\AAA_{\tor}(K_V)$. If $P(G,[\omega])$ is a total toric order, then $P(G,\omega')$ is a total order for each $\omega'\in[\omega]$, and thus $[\omega]$ has precisely $|V|$ elements. Since each $P(G,\omega')$ has exactly one linear extension, total toric orders are indexed by the $(n-1)!$ cyclic equivalence classes of permutations of $V$:
\[
[w]=[(w_1,\dots,w_n)]:=\big\{(w_1,\dots,w_n),(w_2,\dots,w_n,w_1),\dots,(w_n,w_1,\dots,w_{n-1})\big\}\,.
\]
Recall that if $P$ and $P'$ are toric posets over $G$, then $P'$ is an extension of $P$ if $c(P')\subseteq c(P)$. Moreover, $P'$ is a \emph{total toric extension} if $P'$ is a total toric order. Analogous to how a poset is determined by its linear extensions, a toric poset $P$ is determined by its set of total toric extensions, denoted $\LLL_{\tor}(P)$.
\begin{thm}{(\cite{Develin:15}, Proposition 1.7)}\label{thm:Prop1.7}
  Any toric poset $P$ is completely determined by its total toric extensions in the following sense:
  \[
  \overline{c(P)}=\bigcup_{P'\in\LLL_{\tor}(P)}\overline{c(P')}\,.
  \]
\end{thm}

%%---------------------------------------------------------------------------
\subsection{Toric directed paths, chains, and transitivity}\label{subsec:toric-chains}

A \emph{chain} in a poset $P(G,\omega)$ is a totally ordered subset $C\subseteq V$. Equivalently, this means that the elements in $C$ all lie on a common \emph{directed path} $i_1\to i_2\to\cdots\to i_m$ in $\omega$. Transitivity can be characterized in this language: if $i$ and $j$ lie on a common chain, then $i$ and $j$ are comparable in $P(G,\omega)$. Geometrically, $i$ and $j$ being comparable means the hyperplane $H_{i,j}$ does not cut (i.e., is disjoint from) the chamber $c(P(G,\omega))$. 

The toric analogue of a chain is ``essentially'' a totally cyclically ordered set, but care must be taken in the case when $|C|=2$ because every size-two subset  $C\subseteq V$ is trivially totally cyclically ordered. Define a \emph{toric directed path} in $\omega$, to be a directed path $i_1\to i_2\to\cdots\to i_m$ such that the edge $i_1\to i_m$ is also present. We denote such a path by $i_1\to_{\tor}i_m$. Toric directed paths of size $2$ are simply edges, and every singleton set is a toric directed path of size $1$. A fundamental property of toric directed paths is that up to cyclic shifts, they are invariants of toric-equivalence classes. That is, $i_1\to_{\tor}i_m$ is a toric directed path of $\omega$ if and only if each $\omega'\in[\omega]$ has a toric directed path $j_1\to_{\tor} j_m$, for some cyclic shift $(j_1,\dots,j_m)$ in $[(i_1,\dots,i_m)]$. This is Proposition~4.2 of \cite{Develin:15}, and it leads to the notion of a toric chain, which is a totally cyclically ordered subset.

\begin{defn}\label{defn:toric-chain}
  Let $P=P(G,[\omega])$ be a toric poset. A subset $C=\{i_1,\ldots,i_m\} \subseteq V$ is a \emph{toric chain} of $P$ if there exists a cyclic equivalence class $[(i_1,\ldots,i_m)]$ of linear orderings of $C$ with the following property:
  for every $x\in c(P)$ there exists some $(j_1,\dots,j_m)$ in $[(i_1,\ldots,i_m)]$ for which
  \[
  0\leq x_{j_1}<x_{j_2}<\cdots<x_{j_m}<1\,.
  \]
In this situation, we will say that $P|_C=[(i_1,\ldots,i_m)]$. 
\end{defn}

The following is a reformulation of Proposition~6.3 of \cite{Develin:15} using the language of this paper, where notation such as $P(G,\omega)$ and $P(G,[\omega])$ is new.

\begin{prop}\label{prop:toric-chains}
  Fix a toric poset $P=P(G,[\omega])$, and $C=\{i_1,\ldots,i_m\} \subseteq V$.
  The first three of the following four conditions are equivalent,
  and when $m=|C| \neq 2$, they are also equivalent to the fourth.
\begin{enumerate}[(a)]
\item $C$ is a toric chain in $P$, with $P|_C=[(i_1,\ldots,i_m)]$.
\item For every $\omega'\in[\omega]$, the set $C$ is a chain of $P(G,\omega')$, ordered in some cyclic shift of $(i_1,\dots,i_m)$.
\item For every $\omega'\in[\omega]$, the set $C$ occurs as a subsequence of a toric directed path in $\omega'$, in some cyclic shift of the order $(i_i,\dots,i_m)$.
\item Every total toric extension $[w]$ in $\LLL_{\tor}(P)$ has the
same restriction $[w|_C]=[(i_1,\ldots,i_m)]$.
\end{enumerate}
\end{prop}

For ordinary posets, all subsets of chains are chains. The same holds in the toric case.

\begin{prop}
  Subsets of toric chains are toric chains. 
\end{prop}

Having the concept of a toric chain leads to the notion of \emph{toric transitivity}, which is completely analogous to ordinary transitivity when stated geometrically.

\begin{prop}\label{prop:toric-transitivity}
Let $i,j\in V$ be distinct. Then the hyperplane $H_{i,j}^{\tor}$ does not cut the chamber $c(P(G,[\omega]))$ if and only if $i$ and $j$ lie on a common toric chain. 
\end{prop}

%%------------------------------------------------------------------------
\subsection{Toric Hasse diagrams}\label{subsec:toric-hasse}

One of the major drawbacks to studying toric posets combinatorially, as
equivalences of acyclic orientations (rather than geometrically, as
toric chambers), is that a toric poset $P$ or chamber $c=c(P)$
generally arises in multiple toric graphic arrangements
$\AAA_{\tor}(G)$ over the same vertex set. That is, one can have
$P(G,[\omega])=P(G',[\omega'])$ for different graphs, leading to
ambiguity in labeling a toric poset $P$ with a pair $(G,[\omega])$
consisting of a graph $G$ and equivalence class $[\omega]$ in $\Acyc(G)/\!\!\equiv$. 

Toric transitivity resolves this issue. As with ordinary posets, there is a well-defined notion for toric posets of what it means for an edge to be ``implied by transitivity.'' The \emph{toric Hasse diagram} is the graph $\hat{G}^{\torHasse}$ with all such edges removed. In Section~\ref{subsec:toric-hasse}, we encountered the \emph{toric transitive closure}, which is the graph $\bar{G}^{\tor}$ with all such edges included. In other words, given any toric poset $P=P(G,[\omega])$, there is always a unique minimal pair $(\hat{G}^{\torHasse}(P),[\omega^{\torHasse}(P)])$ and maximal pair
$(\bar{G}^{\tor}(P),[\bar{\omega}^{\tor}(P)])$ with the property that the set $c(P)$ is in $\Chambers\AAA_{\tor}(G)$ iff
  \[
  \hat{G}^{\torHasse}(P) \subseteq G  \subseteq \bar{G}^{\tor}(P)
  \]
  where $\subseteq$ is inclusion of edges. In this case, $\omega$ can be taken to be the restriction to $G$ of any orientation in $[\bar{\omega}^{\tor}(P)]$.

Geometrically, the existence of a unique toric Hasse diagram is intuitive; it corresponds to the minimal set of toric hyperplanes that bound the chamber $c(P)$, and the edges implied by transitivity correspond to the additional hyperplanes that do not cut $c(P)$. The technical combinatorial reason for the existence of a unique Hasse diagram (respectively, toric Hasse diagram) follows because the transitive closure (respectively, toric transitive closure) $A\longmapsto\bar{A}$ is a \emph{convex closure}, meaning it
satisfies the following anti-exchange condition;
also see~\cite{Edelman:85}:
\begin{equation*}
  \label{convex-closure-definition}
  \text{ for } a \neq b \text{ with }
  a,b \not\in \bar{A}\text{ and }
  a \in \overline{A \cup \{b\}},
  \text{ one has }b \notin \overline{A \cup \{a\}}\,.
\end{equation*}

Edges $\{i,j\}$ in the Hasse diagram (respectively, toric Hasse diagram) are precisely those whose removal ``change'' the poset (respectively, toric poset), and the geometric definitions make this precise. Though the ordinary and toric cases are analogous, there are a few subtle differences. For example, consider the following ``folk theorem.''

\begin{prop}\label{prop:facets}
  Let $P$ be a poset over $G=(V,E)$ and $\{i,j\}\in E$. Then the following are equivalent:
  \begin{enumerate}[(i)]
      \item The edge $\{i,j\}$ is in the Hasse diagram, $\hat{G}^{\Hasse}(P)$.
      \item Removing $H_{i,j}$ enlarges the chamber $c(P)$.
\item $\overline{c(P)}\cap H_{i,j}$ is a (closed) facet of $P$.
\item The interval $[i,j]$ is precisely $\{i,j\}$.
\item Removing $H_{i,j}$ from $\AAA(G)$ strictly increases the number of chambers.
\end{enumerate}
\end{prop}

Since toric posets are defined geometrically as subsets of $\R^V/\Z^V$ that are chambers of a graphic hyperplane arrangement, the equivalence (i)$\Leftrightarrow$(ii) is immediate for toric posets. Condition~(iii) says that the edges $\{i,j\}$ of the Hasse diagram are precisely the size-2 intervals, and Condition~(iv) says these are the closed face partitions having two blocks of the form $\pi=\{\{i,j\},V\setmin\{i,j\}\}$. Finally, note that the implication (i)$\Rightarrow$(v) in Proposition~\ref{prop:facets} trivially fails for toric posets. For a simple counterexample, take any tree $G$ with at least one edge. Since $\AAA_{\tor}(G)$ has only one chamber, removing any $H_{i,j}^{\tor}$ will never increase the number of chambers. 

Since adding or removing edges implied by toric transitivity does not
change the toric poset, it does not change which sets are toric
chains. Thus, to characterize the toric chains of $P(G,[\omega])$, it
suffices to characterize the toric chains of
$P(\bar{G}^{\tor}(P),[\bar{\omega}^{\tor}(P)])$. The following is immediate.

\begin{rem}
  Let $P$ be a toric poset. A size-$2$ subset $C=\{i,j\}$ of $V$ is a
  toric chain of $P$ if and only if $\{i,j\}$ is an edge of
  $\bar{G}^{\tor}(P)$. In particular, if $C$ is a maximal toric chain,
  then $\{i,j\}$ is an edge in $\hat{G}^{\torHasse}(P)$.
\end{rem}

%%--------------------------------------------------------
\subsection{Toric intervals}\label{subsec:toric-intervals}

To motivate the definition of a toric interval, it helps to first interpret the classical definition in several different ways.

\begin{defn}\label{defn:interval}
    Let $i$ and $j$ be elements of a poset $P=P(G,\omega)$. The \emph{interval} $[i,j]$ is the set $I$ characterized by one of the following equivalent conditions:
    \begin{enumerate}[(i)]
      \item $I=\{k\in V: x_i\leq x_k\leq x_j,\;\text{for all }x\in c(P)\}$;
      \item $I=\{k\in V: k \text{ appears between $i$ and $j$ (inclusive) in any linear extension of $P$}\}$;
      \item $I=\{k\in V: k \text{ lies on a directed path from $i$ to $j$ in $\omega$}\}$.
    \end{enumerate}
Note that if $i=j$, then $[i,j]=\{i\}$, and if $i\nleq j$, then $[i,j]=\varnothing$.
\end{defn}

We will define toric intervals geometrically, motivated by Condition~(i), and show how it is equivalent to the toric version of Condition~(ii). In contrast, Condition (iii) has a small wrinkle -- the property of lying on a directed path from $i$ to $j$ does not depend on the choice of $(G,\omega)$ for $P$. Specifically, if $P(G,\omega)=P(G',\omega')$ and $k$ lies on an $\omega$-directed path from $i$ to $j$, then $k$ lies on an $\omega'$-directed path from $i$ to $j$. This is not the case for toric directed paths in toric posets, as the following example illustrates. As a result, we will formulate and prove a modified version of Condition~(iii) for toric intervals.

\begin{ex}
  Consider the circular graph $G=C_4$, and $\omega\in\Acyc(C_4)$ as shown at left in Figure~\ref{fig:toric-directed-path}. Let $P=P(C_4,[\omega])$, which is a total toric order. Therefore, the toric transitive closure of $(C_4,[\omega])$ is the pair $(K_4,[\omega'])$, where $\omega'$ is shown in Figure~\ref{fig:toric-directed-path} on the right. Therefore, $P(C_4,[\omega])=P(K_4,[\omega'])$. 

  \begin{figure}\centering
    \tikzstyle{to} = [draw,-stealth]
    \tikzstyle{To} = [draw,-stealth,dashed]
    \begin{tikzpicture}[scale=.5]
      \node (4) at (0,6) {$\mathbf{4}$};
      \node (3) at (2,4) {$\mathbf{3}$};
      \node (2) at (2,2) {$\mathbf{2}$};
      \node (1) at (0,0) {$\mathbf{1}$};
      \draw[to] (1) to (2);
      \draw[to] (1) to (4);
      \draw[to] (2) to (3);
      \draw[to] (3) to (4);
      \node at (-3,3) {$\omega\in\Acyc(C_4)$};
    \begin{scope}[shift={(12,0)}]
      \node (4) at (0,6) {$\mathbf{4}$};
      \node (3) at (2,4) {$\mathbf{3}$};
      \node (2) at (2,2) {$\mathbf{2}$};
      \node (1) at (0,0) {$\mathbf{1}$};
      \draw[to] (1) to (2);
      \draw[to] (1) to (3);
      \draw[to] (1) to (4);
      \draw[to] (2) to (3);
      \draw[to] (2) to (4);
      \draw[to] (3) to (4);
      \node at (-3,3) {$\omega'\in\Acyc(K_4)$};
    \end{scope}
    \end{tikzpicture}
    \caption{Despite the equality $P(C_4,[\omega])=P(K_4,[\omega'])$, the set $I=\{1,3\}$ lies on a toric directed path $1\to_{\tor}3$ in $\omega'$ but not for any representative of $[\omega]$.}
    \label{fig:toric-directed-path}
  \end{figure}
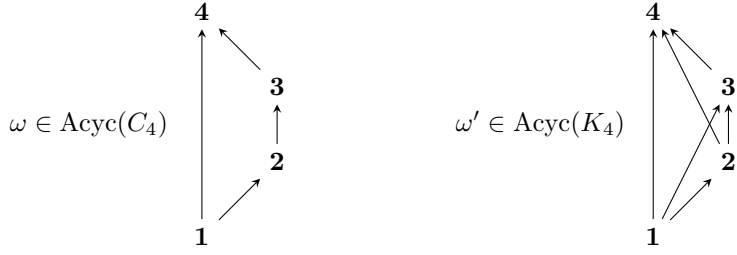

Now, let $i=1$ and $j=3$. The set $\{i,j\}$ lies on a toric directed path from $1$ to $3$ in $\omega'$ (which also contains $2$). However, none of the $4$ representatives in $[\omega]$ contain a toric directed path from $1$ to $3$.
\end{ex}

Another obstacle to formulating the correct toric analogue of an interval is how to characterize which size-$2$ subsets should be toric intervals. This ambiguity arises from the aforementioned ``size-$2$ chain problem'' of all size-$2$ subsets being totally cyclically ordered. Since ordinary intervals are unions of chains, we will require this to be a feature of toric intervals. 

\begin{defn}
    Let $i$ and $j$ be elements of a toric poset $P=P(G,[\omega])$. The \emph{toric interval} $I=[i,j]^{\tor}$ is the empty set if $i,j$ do not lie on a common toric chain, and otherwise is the set
    \begin{equation}\label{eq:toric-interval1}
      I=[i,j]^{\tor}:=\{i,j\}\cup\{k\in V: P|_{\{i,j,k\}}=[(i,k,j)]\}\,.
    \end{equation}
If there is no $k\not\in\{i,j\}$ satisfying Eq.~\eqref{eq:toric-interval1}, then $[i,j]^{\tor}=\{i,j\}$. If $i=j$, then $[i,j]^{\tor}=\{i\}$.
\end{defn}

\begin{rem}\label{rem:toric-interval}
  If $i,j,k$ are distinct elements of the toric interval $I=[i,j]^{\tor}$ of $P=P(G,[\omega])$, then for each $x$ in $c(P)$, exactly one of the following must hold:
\begin{equation}\label{eq:toric-interval2}
0\leq x_i<x_k<x_j<1\,,\qquad 0\leq x_j<x_i<x_k<1\,,\qquad 0\leq x_k<x_j<x_i<1\,.
\end{equation}
\end{rem}

By Theorem~\ref{thm:Prop1.7}, we can rephrase Remark~\ref{rem:toric-interval} as the toric analogue of Definition~\ref{defn:interval}(ii): the toric interval $[i,j]^{\tor}$ in $P$ is the set of elements between $i$ and $j$ in the cyclic order of any total toric extension of $P$.

\begin{cor}
  Suppose $I=[i,j]^{\tor}$ is a toric interval of $P=P(G,[\omega])$ of size $|I|\geq 3$. Then
  \[
  [i,j]^{\tor}=\{i,j\}\cup\{k: [w|_{\{i,j,k\}}]=[(i,k,j)],\;\text{ for all }[w]\in\LLL_{\tor}(P)\}\,.
  \]
\end{cor}

Finally, the toric analogue of Definition~\ref{defn:interval}(iii) can be obtained by first passing to the toric transitive closure.

\begin{prop}\label{prop:toric-interval-iii}
  Fix a toric poset $P=P(G,[\omega])$. An element $k$ is in $[i,j]^{\tor}$ if and only if $k$ lies on a toric directed path $i\to_{\tor} j$ in $\bar{\omega}'$, for some $\bar\omega'\in[\bar{\omega}^{\tor}(P)]$.
\end{prop}

\begin{proof}
  Throughout, let $C=\{i,j,k\}$. Assume that $|C|\geq 3$; the result is trivial otherwise. Suppose $k$ is in $[i,j]^{\tor}$, which means that $P|_C=[(i,k,j)]$. Take any $\omega'\in[\omega]$ for which $i$ is a source. By Proposition~\ref{prop:toric-chains}, the elements of $C$ occur as a subsequence of a toric directed path in $\omega'$, ordered $(i,k,j)$. Since this is a toric chain, the edges $\{i,k\}$, $\{k,j\}$, and $\{i,j\}$ are all implied by toric transitivity. Thus, $k$ lies on a toric directed path $i\to_{\tor}j$ in $\bar{\omega}'$, the unique orientation of $[\bar{\omega}^{\tor}(P)]$ whose restriction to $G$ is $\omega'$.

Conversely, suppose that $k$ lies on a toric directed path $i\to_{\tor} j$ in $\bar\omega'$, for some $\bar\omega'\in[\bar\omega^{\tor}(P)]$. Then $C$ is a toric chain, ordered $P|_C=[(i,k,j)]$, hence $k$ is in $[i,j]^{\tor}$.
\end{proof}

\begin{prop}\label{prop:toric-intervals}
  Let $P=P(G,[\omega])$ be a toric poset. If a set $I\subseteq V$ is a toric interval $I=[i,j]^{\tor}$, then there is some $\omega'\in[\omega]$ for which the set $I$ is the interval $[i,j]$ of $P(G,\omega')$. The converse need not hold.
\end{prop}

\begin{proof}
  Without loss of generality, assume that $G=\bar{G}^{\tor}(P)$. The statement is trivial if $|I|<2$. We need to consider the cases $|I|=2$ and $|I|\geq 3$ separately. In both cases, we will show that one can take $\omega'$ to be any orientation that has $i$ as a source. 
  
  First, suppose $|I|=2$, which means that  $I=\{i,j\}$ is an edge of $G$. Take any $\omega'\in[\omega]$ for which $i$ is a source. Since $|I|=2$, there is no other $k\not\in\{i,j\}$ on a directed path from $i$ to $j$ in $\omega'$, as this would form a toric directed path. Therefore, the interval $[i,j]$ in $P(G,\omega')$ is simply $\{i,j\}$. 

  Next, suppose $|I|\geq 3$. As before, take any $\omega'\in[\omega]$ such that $i$ is a source in $\omega'$. Since  $G=\bar{G}^{\tor}(P)$, the directed edge $i\to j$ is present, and so by Proposition~\ref{prop:toric-interval-iii}, $[i,j]^{\tor}$ consists of all $k\in V$ that lie on a directed path from $i$ to $j$. This is precisely the definition of the interval $[i,j]$ in $P(G,\omega')$. 

To see how the converse can fail, take $G$ to be the line graph on $3$ vertices, and $\omega$ to be the orientation $1\to 2\to 3$. In $P(G,\omega')$, the interval $[1,3]$ is $\{1,2,3\}$ but since $1$ and $3$ do not lie on a common toric chain, $[1,3]^{\tor}=\varnothing$ in $P(G,[\omega])$.
\end{proof}

\begin{prop}\label{prop:Hasse}
  For any toric poset $P=P(G,[\omega])$,
  \begin{equation}\label{eq:inequalities}
  E(\hat{G}^{\Hasse}(P(G,\omega)))\subseteq
  E(\hat{G}^{\torHasse}(P(G,[\omega])))\subseteq E(\bar{G}^{\tor}(P))
  =\bigcap_{\omega'\in[\omega]}E(\bar{G}(P(G,\omega')))\,.
  \end{equation}
\end{prop}

\begin{proof}
  Given the toric Hasse diagram of $P(G,[\omega])$, the ordinary Hasse diagram
  of $P(G,\omega)$ is obtained by removing the edge $\{i_1,i_m\}$ for
  each toric directed path $i_1\to_{\tor} i_m$ in $P(G,\omega)$ of size at least $m\geq 3$. This establishes the first inequality in Eq.~\eqref{eq:inequalities}. 

  The second inequality is obvious. Loosely speaking, the final equality holds because edges in the toric transitive closure are precisely the size-$2$ toric chains, which are precisely the subsets that are size-$2$ chains in every representative poset. We will prove each containment explicitly. For ``$\subseteq$'', take an edge $\{i,j\}$ of $\bar{G}^{\tor}(P)$, which is a size-$2$ toric chain. By Proposition~\ref{prop:toric-chains}, $\{i,j\}$ is a toric chain of $P(G,\omega')$ for all $\omega'\in[\omega]$, which means that it is an edge of the transitive closure $\bar{G}(P(G,\omega'))$. The ``$\supseteq$'' containment is analogous: suppose $\{i,j\}$ is an edge of $\bar{G}(P(G,\omega'))$ for each $\omega'\in[\omega]$. Then by Proposition~\ref{prop:toric-chains}, it is a toric chain of $P$, and hence an edge of $\bar{G}^{\tor}(P)$.
\end{proof}

%%---------------------------
\subsection{Toric antichains}

An {\it antichain} of an ordinary poset $P$ is a subset $A\subseteq
V$ characterized
\begin{enumerate}
\item[$\bullet$] {\it combinatorially} by the condition that no pair
  $\{i,j\}\subseteq A$ with $i\neq j$ are comparable, that is, they lie
  on no common chain of $P$, or
\item[$\bullet$] {\it geometrically} by the equivalent condition that
  the $(|V|-|A|+1)$-dimensional subspace $D_{\pi_A}$ intersects the open
  polyhedral cone $c(P)$ in $\R^V$.
\end{enumerate}

As shown in~\cite{Develin:15}, these two conditions in the toric
setting lead to different notions of toric antichains which are both
easy to formulate. Unlike the case of ordinary posets, these two
definitions are non-equivalent; leading to two distinct versions of a
toric antichains, combinatorial and geometric. The following is the
geometric one which we will use in this paper. Its appearance in
Proposition~\ref{prop:toric-antichains}, which is one of the ``\emph{For
Some}'' structure theorems listed in the Introduction, suggests
that it is the more natural toric analogue of the two.

\begin{defn}
  Given a toric poset $P$ on $V$, say that $A\subseteq
  V$ is a \emph{geometric toric antichain} if $D^{\tor}_{\pi_A}$
  intersects the open toric chamber $c(P)$ in $\R^V/\Z^V$.
\end{defn}

The following characterization of toric antichains was established in~\cite{Develin:15}. It follows because if $D^{\tor}_{\pi_S}$ intersects the open toric chamber $c=c(P)$ in $\Chambers\AAA_{\tor}(G)$, then $D_{\pi_S}$ intersects the open chamber upstairs in $\Chambers\AAA(G)$.

\begin{prop}\label{prop:toric-antichains}
  Let $P=P(G,[\omega])$ be a toric poset. Then a set $A\subseteq V$ is a geometric toric antichain of $P$ if and only if $A$ is an antichain of $P(G,\omega')$ for some $\omega'\in[\omega]$.
\end{prop}

%%=====================================================================
\section{Morphisms of toric posets}\label{sec:toric-morphisms}
%%=====================================================================

Morphisms of ordinary posets have equivalent combinatorial and
geometric characterizations. In contrast, while there seems to be no
simple or obvious combinatorial description for morphisms of toric
posets, the geometric version has a natural toric analogue. 

Firstly, it is clear how to define a \emph{toric isomorphism} between
two toric posets $P$ and $P'$ on vertex sets $V$ and $V'$: a bijection
$\phi\colon V\to V'$ such that the induced isomorphism on
$\R^V/\Z^V\to\R^{V'}/\Z^{V'}$ maps $c(P)$ to $c(P')$ bijectively. The
other types of ordinary poset morphisms have the following toric
analogues:
\begin{itemize}
\item \emph{quotients} that correspond to projecting the toric chamber
  onto a flat of $\AAA_{\tor}(G'_\pi)$ for some closed toric face partition $\pi=\cl_P^{\tor}(\pi)$;
\item \emph{inclusions} that correspond to embedding a toric chamber
  into a higher-dimensional chamber;
\item \emph{extensions} that add relations (toric hyperplanes).
\end{itemize} 

Since every poset morphism can be expressed as the composition of a
quotient, an inclusion, and an extension, it is well-founded to define
a \emph{toric poset morphism} to be the composition of the toric
analogues of these maps. In the remainder of this section, we will describe toric morphisms in detail. Most of the difficulties have already been done in Section~\ref{sec:poset-morphisms}, when interpreting the well-known concept of an ordinary poset morphism geometrically. In contrast, this section is simply an adaptation of this geometric framework from $\R^V$ to $\R^V/\Z^V$, though there are some noticeable differences. For example, there is no toric analogue of intersecting a chamber with a half-space, because the torus minus a hyperplane is connected.

%%-------------------
\subsection{Quotient}

In the ordinary poset case, a quotient is performed by contracting
$P(G,\omega)$ by a partition $\pi=\{B_1,\dots,B_r\}$. Each
$B_i$ gets collapsed into a single vertex, and the resulting acyclic
graph is denoted by $\omega/\!\!\sim_\pi$, which is an element of $\Acyc(G/\!\!\sim_\pi)$. This does not carry over to the
toric case, because in general, contracting a partition will make some
representatives acyclic and others not. However, the geometric
definition has a natural analogue.

Now, let $P=P(G,[\omega])$ be a toric poset, and $\pi$ be any partition of $V$ closed with respect to $P$, i.e., $\pi=\cl_P^{\tor}(\pi)$. By construction, $D_\pi^{\tor}$ is a flat of $\AAA_{\tor}(G'_\pi)$, and so the subset $F_\pi^{\tor}(P)$ is a face of $\AAA_{\tor}(G'_\pi)$. First, we need a map that projects a point $x$ in $c(P)$ onto this face, which is relatively open in the subspace topology of $D^{\tor}_\pi$. This can be extended to the entire torus, by taking the unique map $\bar{d}_\pi$ that makes the following diagram commute, where $d_\pi$ is the mapping from Eq.~\eqref{eq:delta_pi}:
\[
\xymatrix{\R^V\ar[r]^{d_\pi}\ar[d]_q & D_\pi\ar[d]^q
  \\ \R^V/\Z^V\ar@{.>}[r]^{\bar{d}_\pi} & D^{\tor}_\pi}
\]
Explicitly, the map $\bar{d}_\pi$ takes a point $x\in\R^V/\Z^V$, lifts it to a
point $\lift{x}$ in an order polytope in $\R^V$, projects it onto
the flat $D_\pi$ as in Eq.~\eqref{eq:delta_pi}, and then maps
that point down to the toric flat $D^{\tor}_\pi$. In light of this, we will say that the map $\bar{d}_\pi$ is a \emph{projection} onto the toric flat $D^{\tor}_\pi$.

After projecting a chamber $c(P)$ onto a flat $D_\pi^{\tor}$ of $\AAA(G'_\pi)$, we need to project it homeomorphically onto a coordinate subspace of $\R^V/\Z^V$ so it is a chamber of a lower-dimensional toric arrangement. As in the ordinary case, let $W\subseteq V$ be any subset formed by removing all but $1$ coordinate from each $B_i$, and let
$\bar{r}_\pi\colon\R^V/\Z^V\longto\R^W/\Z^W$ be the induced
projection. The $\bar{r}_\pi$-image of $\AAA_{\tor}(G)$ will be the
toric arrangement $\AAA_{\tor}(G/\!\!\sim_\pi)$. As before, the following easily verifiable lemma ensures that our choice of $W\subseteq V$ does not matter.

\begin{lem}\label{lem:toric-homeomorphism}
  Let $\pi=\{B_1,\dots,B_r\}$ be a partition of $V$, and
  $W=\{b_1,\dots,b_r\}\subseteq V$ with $b_i\in B_i$. The restriction
  $\bar{r}_\pi|_{D^{\tor}_\pi}\colon D^{\tor}_\pi\longto\R^W/\Z^W$ is
  a homeomorphism.

  Moreover, all such projection maps for a fixed $\pi$ are
  topologically conjugate in the following sense: If
  $W'=\{b'_1,\dots,b'_r\}\subseteq V$ with $b'_i\in B_i$, and
  projection map $\bar{r}'_\pi|_{D_\pi}\colon D^{\tor}_\pi\longto\R^{W'}/\Z^{W'}$, and
  $\sigma$ is the permutation of $V$ that transposes each $b_i$ with
  $b'_i$, then the following diagram commutes:
  \[
  \xymatrix{D^{\tor}_\pi\ar@{^{(}->}[r] &
    \R^V/\Z^V\ar[r]^{\bar{r}_\pi}\ar[d]_{\bar\Sigma} & \R^W/\Z^W\ar[d]^{\bar\Sigma'}
    \\ D^{\tor}_\pi\ar@{^{(}->}[r] & \R^V/\Z^V\ar[r]^{\bar{r}'_{\pi}} &
    \R^{W'}/\Z^{W'}}
  \]
  Here, $\sigma'\colon W\to W'$ is the map $b_i\mapsto b'_i$, with
  $\bar\Sigma$ and $\bar\Sigma'$ being the induced linear maps as defined
  in Eq.~\eqref{eq:induced}, but done modulo $1$.
\end{lem}

By convexity (in the fundamental affine chambers), two points in the same face of $\AAA_{\tor}(G)$ get mapped to the same face in $\AAA_{\tor}(G/\!\!\sim_\pi)$. In other words, $\bar{d}_\pi$ induces a well-defined map $\bar\delta_\pi$ from $\Faces\AAA_{\tor}(G)$ to $\Faces\AAA_{\tor}(G'_\pi)$ making the following diagram commute:
\[
\xymatrix{\R^V/\Z^V\ar[r]^{\bar{d}_\pi}\ar[d] & D^{\tor}_\pi\ar[d]
  \\ \Faces\AAA_{\tor}(G)\ar@{.>}[r]^{\bar\delta_\pi} &
  \Faces\AAA_{\tor}(G/\!\!\sim_\pi)}
\]
Explicitly, the map $\bar\delta_\pi$ is easiest to defined by the analogous map on closed faces:
\[
\bar\delta_\pi\colon\overline\Faces\,\AAA_{\tor}(G)\longto\overline\Faces\,\AAA_{\tor}(G'_\pi)
\,,\qquad\bar\delta_\pi\colon\overline{c(P)}\cap D_\sigma^{\tor}\longmapsto 
\overline{c(P)}\cap D_\sigma^{\tor}\cap D_\pi^{\tor}\,.
\]
The open faces of $\AAA_{\tor}(G'_\pi)$ are then mapped to faces of the arrangement $\AAA_{\tor}(G/\!\!\sim_\pi)$ under the projection $\bar{r}_\pi|_{D_\pi}\colon D^{\tor}_\pi\longto\R^{W}/\Z^{W}$. Combinatorially, the open faces of $\AAA_{\tor}(G)$ are toric preposets $[\omega]$ over $G$ (i.e., in $\Pre(G)/\!\!\equiv$). These are mapped to toric preposets over $G/\!\!\sim_\pi$ via the composition
\[
\Pre(G)\stackrel{\bar{q}_\pi}{\longto}\Pre(G'_\pi)\stackrel{\bar{p}_\pi}{\longto}\Pre(G/\!\!\sim_\pi)\,,\qquad\qquad
[\omega]\stackrel{\bar{q}_\pi}{\longmapsto}[\omega'_\pi]
\stackrel{\bar{p}_\pi}{\longmapsto}[\omega'_\pi/\!\!\sim_\pi]=[\omega/\!\!\sim_\pi]\,.
\]
The following commutative diagram illustrates the relationship between
the points in $\R^V/\Z^V$, the faces of the toric graphic arrangement
$\AAA_{\tor}(G)$, and the toric preposets over $G$. The left column
depicts the toric preposets over $G$ that are also toric posets.
\begin{equation*}\label{eq:6squares-toric}
  \xymatrix{\R^V/\Z^V\setmin\AAA_{\tor}(G)\,\ar@{^{(}->}[r]\ar[d] &
    \R^V/\Z^V\ar[r]^{\bar{d}_\pi}\ar[d] & D^{\tor}_\pi\ar[r]^{\bar{r}_\pi}\ar[d]
    & \R^W/\Z^W\ar[d]
    \\ \Chambers\AAA_{\tor}(G)\,\ar@{^{(}->}[r]\ar@{<->}[d] &
    \Faces\AAA_{\tor}(G)\ar[r]^{\bar\delta_\pi}\ar@{<->}[d] &
    \Faces\AAA_{\tor}(G'_\pi)\ar[r]^{\bar\rho_\pi}\ar@{<->}[d] &
    \Faces\AAA_{\tor}(G/\!\!\sim_\pi)\ar@{<->}[d]
    \\ \Acyc(G)/\!\!\equiv\,\ar@{^{(}->}[r] &
    \Pre(G)/\!\!\equiv\ar[r]^{\bar{q}_\pi} &
    \Pre(G'_\pi)/\!\!\equiv\ar[r]^{\bar{p}_\pi} & \Pre(G/\!\!\sim_\pi)/\!\!\equiv }
\end{equation*}

To summarize, toric poset morphisms that are quotients are characterized
geometrically by projecting the toric chamber $c(P)$ onto a flat of $\AAA_{\tor}(G'_\pi)$, for some closed toric face partition $\pi=\cl_P^{\tor}(\pi)$. Applying Theorem~\ref{thm:toric-face-partitions} gives a combinatorial interpretation of this, which was not \emph{a priori} obvious.

\begin{cor}\label{cor:toric-morphism}
  Let $P=P(G,[\omega])$ be a toric poset. Contracting $G$ by a partition $\pi\in\Pi_V$ yields a morphism to a toric poset over $G/\!\!\sim_\pi$ if and only if $\omega'/\!\!\sim_\pi$ is acyclic for some orientation $\omega'\in[\omega]$. $\hfill\Box$
\end{cor}

The following is now immediate from Propositions~\ref{prop:toric-intervals} and \ref{prop:toric-antichains}.

\begin{cor}
  Let $P$ be a toric poset over $V$. Then contracting a toric
  interval $I\subseteq V$ or a geometric toric antichain $A\subseteq
  V$ defines a toric morphism. $\hfill\Box$
\end{cor}

%%--------------------
\subsection{Inclusion}

Just like for ordinary posets, a toric poset can be included in larger
one. Let $P$ be a poset over $V$ and let $V\subsetneq V'$. The simplest
injection adds vertices (dimension) but no edges (extra relations). In
this case, the inclusion $\phi\colon P\into P'$ defines a canonical
inclusion $\Phi\colon\R^V/\Z^V\into\R^{V'}/\Z^{V'}$. This sends the
arrangement $\AAA_{\tor}(G)$ in $\R^V$, where $G=(V,E)$, to the same
higher-dimensional arrangement:
\[
\AAA_{\tor}(G\cup(V'\setmin V)):=\{H^{\tor}_{ij}\text{ in
}\R^{V'}:\{i,j\}\in E\}\,.
\]
The toric chamber $c=c(P)$ is sent to the chamber
\[
c(P')=\{x\in\R^{V'}:x_i<x_j\text{ for } \{i,j\}\in E\}\,.
\]

More generally, an injection $P\to P'$ can have added relations in
$P'$ either among the vertices in $P$ or those in $V'\setmin V$.  Such
a map is simply the composition of an inclusion described above and a
toric extension, described below.

%%--------------------
\subsection{Extension}

Extensions of ordinary posets were discussed in
Section~\ref{subsec:extensions}. A poset $P'$ is an extension of $P$
(both assumed to be over the same set $V$) if any of the three
equivalent conditions holds:
\begin{itemize}
  \item $i\leq_P j$ implies $i\leq_{P'} j$;
  \item $\hat{G}^{\Hasse}(P)\subseteq\hat{G}^{\Hasse}(P')$, where
    $\subseteq$ is inclusion of edge sets;
  \item $c(P')\subseteq c(P)$.
\end{itemize}

The first of these conditions does not carry over nicely to the toric
setting, but the second two do. A toric poset $P'$ is a \emph{toric
  extension} of $P$ if and only one has an inclusion of their open
polyhedral cones $c(P')\subseteq c(P)$ in $\R^V/\Z^V$, which is
equivalent to
$\hat{G}^{\torHasse}(P)\subseteq\hat{G}^{\torHasse}(P')$.

%%------------------
\subsection{Summary}

Up to isomorphism, every toric poset morphism $P=P(G,[\omega])$ can be decomposed into a sequence of three steps:
\begin{enumerate}[(i)]
\item \emph{quotient}: Collapsing $G$ by a partition $\pi$ that preserves acyclicity of some $\omega'\in[\omega]$ (projecting to a flat $D_\pi^{\tor}$ of $\AAA(G'_\pi)$ for some partition $\pi=\cl_P^{\tor}(\pi)$).
\item \emph{inclusion}: Adding vertices (adding dimensions).
\item \emph{extension}: Adding relations (cutting the chamber with
  toric hyperplanes).
\end{enumerate}
Note that in the special case of the morphism $P\longto P'$ being
surjective, the inclusion step is eliminated and the entire process
can be described geometrically by projecting $\overline{c(P)}$ to a
toric flat $D^{\tor}_\pi$ and a then adding toric hyperplanes.

%%=====================================================================
\section{Toric order ideals and filters}\label{sec:toric-ideals}
%%=====================================================================

 Let $P$ be a poset over a set $V$ of size at least $2$, and suppose $\phi\colon P\to P'$ is a morphism to a poset over a size-$2$ subset $V'\subseteq V$. This is achieved by projecting $c(P)$ onto a flat $D_\pi$ of $\AAA(G'_{\pi})$ such that $\pi=\cl_P(\pi)$ has at most two blocks, and hence $\bar{F}_\pi=\overline{c(P)}\cap D_\pi$ is at most $2$-dimensional. A point $x=(x_1,\dots,x_n)$ on $F_\pi$ has at most two distinct entries. Thus, the partition $\pi=\{I,J\}$ of $V$ satisfies
\begin{itemize}
\item $x_{i_k}=x_{i_\ell}$ for all $i_k,i_\ell$ in $I$;
\item $x_{j_k}=x_{j_\ell}$ for all $j_k,j_\ell$ in $J$;
\item $x_i\leq x_j$ for all $i\in I$ and $j\in J$.
\end{itemize}
The set $I$ is called an \emph{order ideal} or just an \emph{ideal} of
$P$ and $J$ is called a \emph{filter}. 

Ideal/filter pairs are thus characterized by closed partitions $\pi$ of $V$ such that $D_\pi$ intersects $\overline{c(P)}$ in at most two dimensions. The set of ideals has a natural poset structure by subset
inclusion. Allowing $I$ or $J$ to be empty, this poset has a
unique maximal element $I=V$ (corresponding to $J=\varnothing$) and
minimal element $I=\varnothing$ (corresponding to $J=V$). Moreover,
the order ideal poset is a lattice; this is
well-known~\cite{Stanley:01}. Similarly, the set of filters is a
lattice as well.

Toric order ideals and filters can be defined similarly.

\begin{defn}
  Let $P$ be a toric poset over $V$, and suppose $\phi\colon P\to
  P'$ is a morphism to a toric poset over a size-$2$
  subset $V'\subseteq V$. This projects $c(P)$ onto a toric flat
  $D^{\tor}_\pi$ of $\AAA_{\tor}(G'_\pi)$ for some $\pi=\cl_P^{\tor}(\pi)$ such that
  $\bar{F}^{\tor}_\pi=\overline{c(P)}\cap D^{\tor}_\pi$ is at most $2$-dimensional.
  For the partition $\pi=\{I,J\}$ of $V$, each point
  $x=(x_1,\dots,x_n)$ on $F^{\tor}_\pi$ satisfies
  \begin{itemize}
  \item $\modone{x_{i_k}}=\modone{x_{i_\ell}}$ for all $i_k,i_\ell$ in $I$;
  \item $\modone{x_{j_k}}=\modone{x_{j_\ell}}$ for all $j_k,j_\ell$ in $J$.
  \end{itemize}
  The set $I$ is called a \emph{toric order ideal} of $P$.
\end{defn}

\begin{rem}\label{rem:toric-filter}
  By symmetry, if $I$ is a toric order ideal, then so is $J:=V\setmin I$. A \emph{toric filter} can be defined analogously, and it is clear that these two concepts are identical. Henceforth, we will stick with the term ``toric filter'' to avoid ambiguity with the well-established but unrelated notion of a toric ideal from commutative algebra and algebraic geometry \cite{sturmfels1996grobner}.
\end{rem}

By construction, toric filters are characterized by closed toric partitions $\pi$ of $V$ such that $D_\pi^{\tor}$ intersects $\overline{c(P)}$ in at most two dimensions -- either a two-dimensional face of $P$ or of an extension $P'$ over $G'_\pi$.

\begin{prop}\label{prop:toric-ideals}
  Let $P(G,[\omega])$ be a toric poset. The following are equivalent for a subset $I\subseteq V$.
  \begin{enumerate}[(i)]
    \item $I$ is a toric filter of $P(G,[\omega])$;
    \item $I$ is an ideal of $P(G,\omega')$ for some $\omega'\in[\omega]$;
    \item $I$ is a filter of $P(G,\omega'')$ for some $\omega''\in[\omega]$;
    \item In at least one total toric extension of $P(G,[\omega])$, the
      elements in $I$ appear in consecutive cyclic order.
  \end{enumerate}
\end{prop}

\begin{proof}
  The result is obvious if $I=\varnothing$ or $I=V$, so assume that  $\varnothing\subsetneq I\subsetneq V$, and $\pi=\{I,V\setmin I\}$.
  This forces $D^{\tor}_\pi$ to be two-dimensional (rather than
  one-dimensional). 

  (i)$\Rightarrow$(ii): If $I$ is a toric filter of $P(G,[\omega])$,
  then $D_\pi^{\tor}$ intersects $c(P(G,[\omega]))$ in
  two-dimensions, and so $D_\pi$ intersects an order polytope
  $\O(P(G,\omega'))$ in two-dimensions, for some
  $\omega'\in[\omega]$. Therefore, $D_\pi$ intersects the chamber $c(P(G,\omega'))$ in two-dimensions, and hence $I$ is an ideal of $P(G,\omega')$.

  (ii)$\Rightarrow$(i): Suppose that $I$ is an ideal of
  $P(G,\omega')$ for $\omega'\in[\omega]$. Then 
  $\bar{F}_\pi=\overline{c(P(G,\omega'))}\cap D_\pi$ is two-dimensional, and it descends to a two-dimensional face $\bar{F}^{\tor}_\pi$ of the toric poset
  $P(G,[\omega'])=P(G,[\omega])$ or of some extension (if $\bar{F}^{\tor}$ intersects the interior). Therefore, $I$ is a toric filter of
  $P(G,[\omega])$.

  (ii)$\Leftrightarrow$(iii): Immediate by Remark~\ref{rem:toric-filter} upon reversing the roles of $I$ and $V\setmin I$. 

  (ii)$\Rightarrow$(iv): If $I$ is a size-$k$ ideal of
  $P(G,\omega')$, then by a well-known property of posets, there is a linear extension of the form
  $(i_1,\dots,i_k,v_{k+1},\dots,v_n)$, where each $i_j\in I$. The
  cyclic equivalence class $[(i_1,\dots,i_k,v_{k+1},\dots,v_n)]$ is
  a total toric extension of $P(G,[\omega'])=P(G,[\omega])$ in
  which the elements of $I$ appear in consecutive cyclic order.
  
  (iv)$\Rightarrow$(ii): Suppose $[(i_1,\dots,i_k,v_{k+1},\dots,v_n)]$
  is a toric total extension of $P(G,[\omega])$. This means that for
  some $x\in\R^V/\Z^V$,
  \begin{equation}\label{eq:total-order}
    0\leq x_{i_1}<\cdots<x_{i_k}<x_{v_{k+1}}<\cdots<x_{v_n}<1\,.
  \end{equation}
  The unique preimage $\lift{x}$ of this point in $[0,1)^V$ under
  the quotient map $q\colon\R^V\to\R^V/\Z^V$ is in the 
  order polytope of some $P(G,\omega')$ that maps into
  $\overline{c(P(G,[\omega']))}=\overline{c(P(G,[\omega]))}$.
  Since the coordinates of $\lift{x}$ are totally ordered as
  in Eq.~\eqref{eq:total-order}, $(i_1,\dots,i_k,v_{k+1},\dots,v_n)$
  is a linear extension of $P(G,\omega')$. Moreover, the coordinates in $I$ form an initial segment of this linear extension, hence $I$ is an ideal of $P(G,\omega')$.

  (iii)$\Leftrightarrow$(iv): Immediate by Remark~\ref{rem:toric-filter} upon reversing the roles of $I$ and $V\setmin I$. 
\end{proof}

Proposition~\ref{prop:toric-ideals} along with a result of Stanley gives a nice characterization of the toric filters in terms of vertices of order polytopes. Every filter $I$ of a poset $P$ has a \emph{characteristic function}
\[
\chi_I\colon P\longto\R\,,\qquad \chi_I(k)=\begin{cases}1,& k\in I \\ 0, & k\not\in I.\end{cases}
\]
We identify $\chi_I(P)$ with the corresponding vector in $\{0,1\}^V\subseteq\R^V$.

\begin{prop}{(\cite{Stanley:86}, Corollary 1.3)}
  Let $P$ be a poset. The vertices of the order polytope $\O(P)$ are the characteristic functions $\chi_I$ of filters of $P$. 
\end{prop}

Given a toric poset, we can define the characteristic function $\chi_I$ of a toric filter similarly. However, one must be careful because under the canonical quotient to the torus, the vertices of every order polytope get identified to $(0,\dots,0)$. Therefore, we will still identify $\chi_I$ with a point in $\R^V$, not $\R^V/\Z^V$. 

\begin{cor}
  Let $P=P(G,[\omega])$ be a toric poset and $I\subseteq V$. Then $\chi_I$ is the characteristic function of a toric filter of $P$ if and only if $\chi_I$ is the vertex of $\O(P(G,\omega'))$ for some $\omega'\in[\omega]$. 
\end{cor}

Let $J_{\tor}(P)$ denote the set of toric filters of $P$. This has a natural poset structure by subset inclusion. Once again, there is a unique maximal element $I=V$ and minimal element $I=\varnothing$.  

\begin{prop}
  With respect to subset inclusion and cardinality rank function, $J_{\tor}(P)$ is a graded poset.
\end{prop}

\begin{proof}
  Let $P=P(G,[\omega])$ be a toric poset over $G=(V,E)$. It suffices to show that every nonempty toric filter $J$ contains a toric filter $J'$ of cardinality $|J'|=|J|-1$.

By Proposition~\ref{prop:toric-ideals}, the set $J$ is an order ideal of $P'=P(G,\omega')$ for some $\omega'\in[\omega]$. Choose any minimal element $v\in V$ of $P'$, which is a source of $\omega'$. Let $\omega''$ be the orientation obtained by flipping $v$ into a sink. The set $J':=J\setmin\{v\}$ is an ideal of $P(G,\omega'')$, and so by Proposition~\ref{prop:toric-ideals}, it is a toric filter of $P(G,[\omega''])=P(G,[\omega])$. 
\end{proof}

\begin{ex}\label{ex:C4-1234}
  Let $G=C_4$, the circle graph on $4$ vertices, and let
  $\omega\in\Acyc(G)$ be the orientation shown at left in
  Figure~\ref{fig:C4-1234}. The Hasse diagram of the poset $P=P(G,\omega)$
  is a line graph $\hat{G}^{\Hasse}(P)=L_4$, and the transitive closure is $\bar{G}(P)=K_4$. Since $V$ is a size-$4$ toric chain, it is totally cyclically ordered in
  every $\omega'\in[\omega]$, and the dashed edges are additionally
  implied by toric transitivity. Thus,
  \[
  L_4=\hat{G}^{\Hasse}(P(G,\omega))\subsetneq\hat{G}^{\torHasse}(P(G,[\omega]))=C_4\,,\qquad
  \bar{G}(P(G,\omega))=\bar{G}^{\tor}(P(G,[\omega]))=K_4\,.
  \]
  The $4$ torically equivalent orientations are shown in
  Figure~\ref{fig:C4-1234}. The only total toric extension of
  $P(G,[\omega])$ is 
  \[
  [(1,2,3,4)]=\{(1,2,3,4)\,,\;(2,3,4,1)\,,\;(3,4,1,2)\,,\;(4,1,2,3)]\,,
  \]
  and this is shown at right in Figure~\ref{fig:C4-1234-ideals}.
  The toric filters are all subsets of $V$ that appear as an initial
  segment in one of these four total orders. The poset
  $J_{\tor}(P(C_4,[\omega]))$ is shown at left in
  Figure~\ref{fig:C4-1234-ideals}. Note that unlike the ordinary poset
  case, it is not a lattice.

  \begin{figure}\centering
    \tikzstyle{to} = [draw,-stealth]
    \tikzstyle{To} = [draw,-stealth,dashed]
    \begin{tikzpicture}[scale=.5]
      \node (4) at (0,6) {$\mathbf{4}$};
      \node (3) at (2,4) {$\mathbf{3}$};
      \node (2) at (2,2) {$\mathbf{2}$};
      \node (1) at (0,0) {$\mathbf{1}$};
      \draw[to] (1) to (2);
      \draw[To] (1) to (3);
      \draw[to] (1) to (4);
      \draw[to] (2) to (3);
      \draw[To] (2) to (4);
      \draw[to] (3) to (4);
      \node at (-3.2,3) {$\omega\in\Acyc(C_4)$};
      \begin{scope}[shift={(5,0)}]
        \node at (-1.5,3) {$\equiv$};
        \node (4) at (0,6) {$\mathbf{1}$};
        \node (3) at (2,4) {$\mathbf{4}$};
        \node (2) at (2,2) {$\mathbf{3}$};
        \node (1) at (0,0) {$\mathbf{2}$};
        \draw[to] (1) to (2);
        \draw[To] (1) to (3);
        \draw[to] (1) to (4);
        \draw[to] (2) to (3);
        \draw[To] (2) to (4);
        \draw[to] (3) to (4);
      \end{scope}
      \begin{scope}[shift={(10,0)}]
        \node at (-1.5,3) {$\equiv$};
        \node (4) at (0,6) {$\mathbf{2}$};
        \node (3) at (2,4) {$\mathbf{1}$};
        \node (2) at (2,2) {$\mathbf{4}$};
        \node (1) at (0,0) {$\mathbf{3}$};
        \draw[to] (1) to (2);
        \draw[To] (1) to (3);
        \draw[to] (1) to (4);
        \draw[to] (2) to (3);
        \draw[To] (2) to (4);
        \draw[to] (3) to (4);
      \end{scope}
      \begin{scope}[shift={(15,0)}]
      \node at (-1.5,3) {$\equiv$};
      \node (4) at (0,6) {$\mathbf{3}$};
      \node (3) at (2,4) {$\mathbf{2}$};
      \node (2) at (2,2) {$\mathbf{1}$};
      \node (1) at (0,0) {$\mathbf{4}$};
      \draw[to] (1) to (2);
      \draw[To] (1) to (3);
      \draw[to] (1) to (4);
      \draw[to] (2) to (3);
      \draw[To] (2) to (4);
      \draw[to] (3) to (4);
      \end{scope}
    \end{tikzpicture}
    \caption{The four torically equivalent orientations to
      $\omega\in\Acyc(C_4)$, shown at left. The edges
      implied by toric transitivity are dashed.}\label{fig:C4-1234}
  \end{figure}
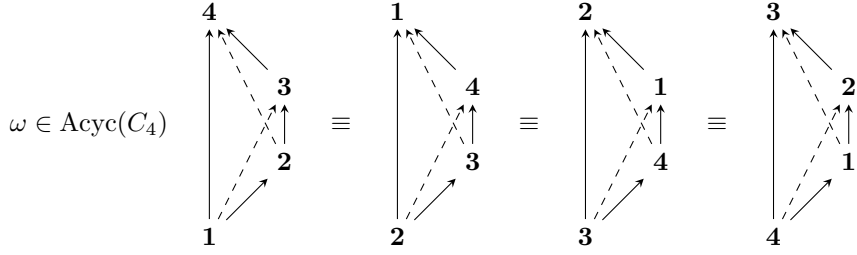

  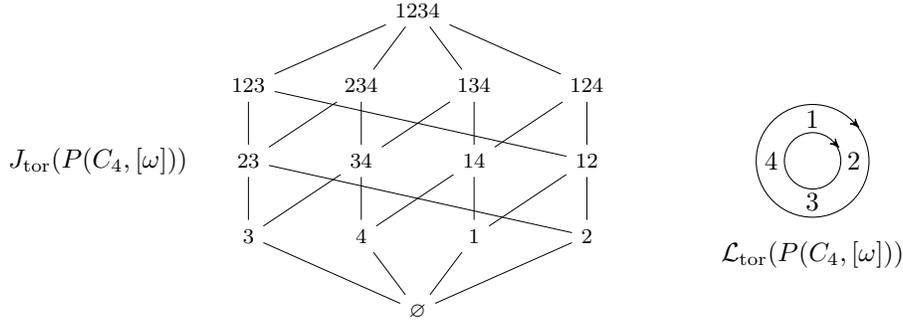
\begin{figure}\centering
    \tikzstyle{to} = [draw,-stealth]
    \tikzstyle{To} = [draw,-stealth,dashed]
    \begin{tikzpicture}[scale=.5]
      \begin{scope}[shift={(0,0)}]
        \node (1234) at (0,6) {\footnotesize $1234$};
        \node (123) at (-4.5,4) {\footnotesize $123$};
        \node (234) at (-1.5,4) {\footnotesize $234$};
        \node (134) at (1.5,4) {\footnotesize $134$};
        \node (124) at (4.5,4) {\footnotesize $124$};
        \node (23) at (-4.5,2) {\footnotesize $23$};
        \node (34) at (-1.5,2) {\footnotesize $34$};
        \node (14) at (1.5,2) {\footnotesize $14$};
        \node (12) at (4.5,2) {\footnotesize $12$};
        \node (3) at (-4.5,0) {\footnotesize $3$};
        \node (4) at (-1.5,0) {\footnotesize $4$};
        \node (1) at (1.5,0) {\footnotesize $1$};
        \node (2) at (4.5,0) {\footnotesize $2$};
        \node (null) at (0,-2) {\footnotesize $\varnothing$};
        \draw[very thin] (1234) -- (123) -- (23) -- (3) -- (null);
        \draw[very thin] (1234) -- (234) -- (34) -- (4) -- (null);
        \draw[very thin] (1234) -- (134) -- (14) -- (1) -- (null);
        \draw[very thin] (1234) -- (124) -- (12) -- (2) -- (null);
        \draw[very thin] (123) -- (12);
        \draw[very thin] (234) -- (23);
        \draw[very thin] (23) -- (2);
        \draw[very thin] (12) -- (1);
        \draw[very thin] (124) -- (14) -- (4);
        \draw[very thin] (134) -- (34) -- (3);  
        \node at (-8.5,2) {$J_{\tor}(P(C_4,[\omega]))$};
      \end{scope}
      \begin{scope}[shift={(10.5,2)},>=stealth']
        \draw[decoration={markings, mark=at position 0.125 with
            {\arrow{<}}}, postaction={decorate}] (0,0) circle(1.5cm);
        \draw[decoration={markings, mark=at position 0.125 with
            {\arrow{<}}}, postaction={decorate}] 
        (0cm,0cm) circle(0.75cm); 
        \draw (0,1.1) node{$1$}; 
        \draw (1.1,0) node{$2$}; 
        \draw (0,-1.1) node{$3$}; 
        \draw (-1.1,0) node{$4$};
        \node at (0,-2.5) {$\LLL_{\tor}(P(C_4,[\omega]))$};
      \end{scope}
    \end{tikzpicture}
    \caption{The toric filters of the toric poset $P(C_4,[\omega])$ form a
      poset that is not a lattice. The vertices should be thought of as subsets of $\{1,2,3,4\}$; order does not matter. That is, $134$ represents 
      $\{1,3,4\}$.}\label{fig:C4-1234-ideals}
  \end{figure}
 
\end{ex}

\begin{ex}\label{ex:C4-1324}
  Let $G=C_4$, as in Example~\ref{ex:C4-1234}, but now let
  $\omega'\in\Acyc(G)$ be the orientation shown at left in
  Figure~\ref{fig:C4-1324}. The only nonempty toric chains are the four
  vertices (size 1) and the four edges (size 2). Since $\omega'$ has no toric chains of
  size greater than $2$, the Hasse diagram and the transitive closure
  of the toric poset $P(G,[\omega'])$ are both $C_4$. Note that the
  transitive closure of the (ordinary) poset $P(G,\omega')$ contains
  the edge $\{1,3\}$, and so as graphs,
  $\bar{G}(P(G,\omega'))\neq\bar{G}^{\tor}(P(G,[\omega']))$. The $6$ torically
  equivalent orientations of $\omega'$ are shown in
  Figure~\ref{fig:C4-1324}. There are four total toric extensions of
  $P(G,[\omega'])$ which are shown on the right in
  Figure~\ref{fig:C4-1324-ideals}, as cyclic words. The toric filters
  are all subsets of $V$ that appear as a consecutive segment in one
  of these four total orders. The poset $J_{\tor}(P(C_4,[\omega']))$
  of toric filters is shown at left in
  Figure~\ref{fig:C4-1324-ideals}. In this particular case, the poset
  of toric filters is a lattice. In fact, it is isomorphic to a Boolean lattice, because every subset of $\{1,2,3,4\}$ appears consecutively (ignoring relative order) in one of the four cyclic words in Figure~\ref{fig:C4-1324-ideals}.

  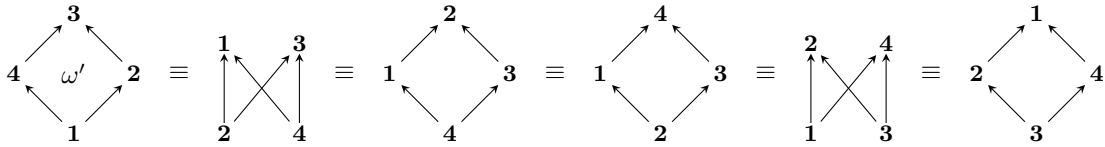
\begin{figure}\centering
    \tikzstyle{to} = [draw,-stealth]
    \tikzstyle{To} = [draw,-stealth,dashed]
    \begin{tikzpicture}[scale=.4,shorten >= -2.5pt, shorten <= -2.5pt]
      \node (3) at (0,4) {\small $\mathbf{3}$};
      \node (4) at (-2,2) {\small $\mathbf{4}$};
      \node (2) at (2,2) {\small $\mathbf{2}$};
      \node (1) at (0,0) {\small $\mathbf{1}$};
      \draw[to] (1) to (2);
      \draw[to] (1) to (4);
      \draw[to] (2) to (3);
      \draw[to] (4) to (3);
      \node at (0,2) {$\omega'$};
      \node at (3.5,2) {$\equiv$};
      \begin{scope}[shift={(5,0)}]
        \node (2) at (0,0) {\small $\mathbf{2}$};
        \node (1) at (0,3) {\small $\mathbf{1}$};
        \node (3) at (2.5,3) {\small $\mathbf{3}$};
        \node (4) at (2.5,0) {\small $\mathbf{4}$};
        \draw[to] (2) to (1);
        \draw[to] (4) to (1);
        \draw[to] (2) to (3);
        \draw[to] (4) to (3);
        \node at (4,2) {$\equiv$};
      \end{scope}
      \begin{scope}[shift={(12.5,0)}]
        \node (2) at (0,4) {\small $\mathbf{2}$};
        \node (1) at (-2,2) {\small $\mathbf{1}$};
        \node (3) at (2,2) {\small $\mathbf{3}$};
        \node (4) at (0,0) {\small $\mathbf{4}$};
        \draw[to] (4) to (1);
        \draw[to] (4) to (3);
        \draw[to] (1) to (2);
        \draw[to] (3) to (2);
        \node at (3.5,2) {$\equiv$};
      \end{scope}
      \begin{scope}[shift={(19.5,0)}]
        \node (4) at (0,4) {\small $\mathbf{4}$};
        \node (1) at (-2,2) {\small $\mathbf{1}$};
        \node (3) at (2,2) {\small $\mathbf{3}$};
        \node (2) at (0,0) {\small $\mathbf{2}$};
        \draw[to] (1) to (4);
        \draw[to] (3) to (4);
        \draw[to] (2) to (1);
        \draw[to] (2) to (3);
        \node at (3.5,2) {$\equiv$};
      \end{scope}
      \begin{scope}[shift={(24.5,0)}]
        \node (1) at (0,0) {\small $\mathbf{1}$};
        \node (2) at (0,3) {\small $\mathbf{2}$};
        \node (4) at (2.5,3) {\small $\mathbf{4}$};
        \node (3) at (2.5,0) {\small $\mathbf{3}$};
        \draw[to] (1) to (2);
        \draw[to] (1) to (4);
        \draw[to] (3) to (2);
        \draw[to] (3) to (4);
        \node at (4,2) {$\equiv$};
      \end{scope}
      \begin{scope}[shift={(32,0)}]
        \node (1) at (0,4) {\small $\mathbf{1}$};
        \node (2) at (-2,2) {\small $\mathbf{2}$};
        \node (4) at (2,2) {\small $\mathbf{4}$};
        \node (3) at (0,0) {\small $\mathbf{3}$};
        \draw[to] (3) to (2);
        \draw[to] (3) to (4);
        \draw[to] (2) to (1);
        \draw[to] (4) to (1);
      \end{scope}
  \end{tikzpicture}
  \caption{The six torically equivalent orientations to
    $\omega'\in\Acyc(C_4)$, shown at left.}\label{fig:C4-1324}
\end{figure}

\begin{figure}\centering
  \tikzstyle{to} = [draw,-stealth]
  \tikzstyle{To} = [draw,-stealth,dashed]
  \begin{tikzpicture}[scale=.5,>=stealth']
    \begin{scope}[shift={(0,0)}]
      \node (1324) at (0,6) {\footnotesize $1324$};
      \node (132) at (-4.5,4) {\footnotesize $132$};
      \node (124) at (-1.5,4) {\footnotesize $124$};
      \node (134) at (1.5,4) {\footnotesize $134$};
      \node (324) at (4.5,4) {\footnotesize $243$};
      \node (12) at (-7.5,2) {\footnotesize $12$};
      \node (13) at (-4.5,2) {\footnotesize $13$};
      \node (14) at (-1.5,2) {\footnotesize $14$};
      \node (23) at (1.5,2) {\footnotesize $23$};
      \node (24) at (4.5,2) {\footnotesize $24$};
      \node (34) at (7.5,2) {\footnotesize $34$};
      \node (1) at (-4.5,0) {\footnotesize $1$};
      \node (2) at (-1.5,0) {\footnotesize $2$};
      \node (3) at (1.5,0) {\footnotesize $3$};
      \node (4) at (4.5,0) {\footnotesize $4$};
      \node (null) at (0,-2) {\footnotesize $\varnothing$};
      \draw[very thin] (1324) -- (124);
      \draw[very thin] (124) -- (12);
      \draw[very thin] (124) -- (14);
      \draw[very thin] (124) -- (24);
      \draw[very thin] (1324) -- (132);
      \draw[very thin] (132) -- (13);
      \draw[very thin] (132) -- (12);
      \draw[very thin] (132) -- (23);
      \draw[very thin] (1324) -- (134);
      \draw[very thin] (134) -- (13);
      \draw[very thin] (134) -- (14);
      \draw[very thin] (134) -- (34);
      \draw[very thin] (1324) -- (324);
      \draw[very thin] (324) -- (23);
      \draw[very thin] (324) -- (24);
      \draw[very thin] (324) -- (34);
      \draw[very thin] (12) -- (1);
      \draw[very thin] (12) -- (2);
      \draw[very thin] (13) -- (1);
      \draw[very thin] (13) -- (3);
      \draw[very thin] (14) -- (1);
      \draw[very thin] (14) -- (4);
      \draw[very thin] (23) -- (2);
      \draw[very thin] (23) -- (3);
      \draw[very thin] (24) -- (2);
      \draw[very thin] (24) -- (4);
      \draw[very thin] (34) -- (3);
      \draw[very thin] (34) -- (4);
      \draw[very thin] (1) -- (null);
      \draw[very thin] (2) -- (null);
      \draw[very thin] (3) -- (null);
      \draw[very thin] (4) -- (null);
      \node at (-6,-2) {$J_{\tor}(P(C_4,[\omega']))$};
    \end{scope}
    \begin{scope}[shift={(13,4.5)}]
      \draw[decoration={markings, mark=at position 0.125 with
          {\arrow{<}}}, postaction={decorate}] (0,0) circle(1.5cm);
      \draw[decoration={markings, mark=at position 0.125 with
          {\arrow{<}}}, postaction={decorate}] 
      (0cm,0cm) circle(0.75cm); 
      \draw (0,1.1) node{$1$}; 
      \draw (1.1,0) node{$3$}; 
      \draw (0,-1.1) node{$2$}; 
      \draw (-1.1,0) node{$4$};
    \end{scope}
    \begin{scope}[shift={(13,.5)}]
      \draw[decoration={markings, mark=at position 0.125 with
          {\arrow{<}}}, postaction={decorate}] (0,0) circle(1.5cm);
      \draw[decoration={markings, mark=at position 0.125 with
          {\arrow{<}}}, postaction={decorate}] 
      (0cm,0cm) circle(0.75cm); 
      \draw (0,1.1) node{$1$}; 
      \draw (1.1,0) node{$3$}; 
      \draw (0,-1.1) node{$4$}; 
      \draw (-1.1,0) node{$2$};
      \node at (2.5,-2.5) {$\LLL_{\tor}(P(C_4,[\omega']))$};
    \end{scope}
    \begin{scope}[shift={(18,4.5)}]
      \draw[decoration={markings, mark=at position 0.125 with
          {\arrow{<}}}, postaction={decorate}] (0,0) circle(1.5cm);
      \draw[decoration={markings, mark=at position 0.125 with
          {\arrow{<}}}, postaction={decorate}] 
        (0cm,0cm) circle(0.75cm); 
        \draw (0,1.1) node{$3$}; 
        \draw (1.1,0) node{$1$}; 
        \draw (0,-1.1) node{$2$}; 
        \draw (-1.1,0) node{$4$};
      \end{scope}
      \begin{scope}[shift={(18,.5)}] 
        \draw[decoration={markings, mark=at position 0.125 with
            {\arrow{<}}}, postaction={decorate}] (0,0) circle(1.5cm);
        \draw[decoration={markings, mark=at position 0.125 with
            {\arrow{<}}}, postaction={decorate}] 
        (0cm,0cm) circle(0.75cm); 
        \draw (0,1.1) node{$3$}; 
        \draw (1.1,0) node{$1$}; 
        \draw (0,-1.1) node{$4$}; 
        \draw (-1.1,0) node{$2$};
      \end{scope}
  \end{tikzpicture}
  \caption{The toric filters of the toric poset $P(C_4,[\omega'])$ form
    a poset that happens to be a lattice. Each toric filter appears as a
    consecutive sequence in one of the total toric extensions, shown
    at right.}\label{fig:C4-1324-ideals}
\end{figure}
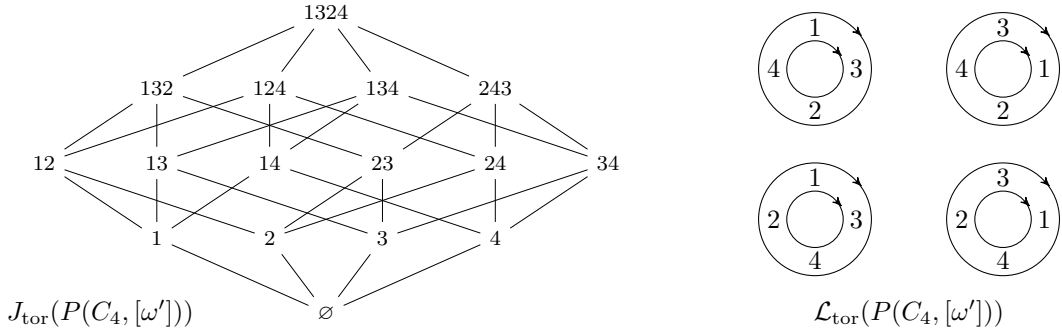

\end{ex}

%%=====================================================================
\section{Application to Coxeter groups}\label{sec:coxeter}
%%=====================================================================

A Coxeter system is a pair $(W,S)$ consisting of a \emph{Coxeter group} $W$
generated by a finite set of involutions $S=\{s_1,\dots,s_n\}$ with presentation
\[
W=\<S\mid s_i^2=1,\;(s_is_j)^{m_{i,j}}=1\>
\]
where $2\leq m_{i,j}\leq\infty$ for $i\neq j$. The corresponding
Coxeter graph $\Gamma$ has vertex set $V=S$ and edges $\{i,j\}$ for
each $m_{i,j}\geq 3$ labeled with $m_{i,j}$ (label usually omitted if
$m_{i,j}=3$). A \emph{Coxeter element} is the product of the
generators in some order, and every Coxeter element $c\in W$ defines a
partial ordering on $S$ via an acyclic orientation
$\omega(c)\in\Acyc(\Gamma)$: Orient $s_i\to s_j$ iff $s_i$ precedes
$s_j$ in some (equivalently, every) reduced expression for
$c$. Conjugating a Coxeter element by an initial generator (note that
$s_i=s_i^{-1}$) cyclically shifts it:
\[
s_{x_1}(s_{x_1}s_{x_2}\cdots s_{x_n})s_{x_1}=s_{x_2}\cdots s_{x_n}s_{x_1}\,,
\]
and the corresponding acyclic orientation differs by reversing the
orientations of all edges incident to $s_{x_1}$, thereby converting it
from a source to a sink vertex. In 2009, H.~ and K.~Eriksson
showed~\cite{Eriksson:09} that two Coxeter elements $c$ and $c'$ are
conjugate if and only if $\omega(c)\equiv\omega(c')$. Thus, there are
bijections between the set $\Cox(W)$ of Coxeter elements and
$\Acyc(\Gamma)$, as well as between the corresponding conjugacy
classes and the toric equivalence classes:
\begin{align*}%\addtolength{\jot}{1em}
  \def\arraystretch{1.5}
  \begin{array}{cllllc}
    \Cox(W)\longto\Acyc(\Gamma) &&&&&
    \Conj(\Cox(W))\longto\Acyc(\Gamma)/\!\!\equiv \\
    \,\,c\longmapsto\omega(c) &&&&& \,\,\cl_W(c)\longmapsto[\omega(c)]
  \end{array}
\end{align*}
As an example, suppose $W=W(\widetilde{A}_3)$, the affine Coxeter
group of type $A$, and let $c=s_1s_2s_3s_4$ (using $s_4$ instead of
the usual $s_0$), as shown below:
\begin{center}
  \tikzstyle{to} = [draw,-stealth]
  \tikzstyle{e} = [draw]
  \begin{tikzpicture}[scale=.7,shorten >= -2pt, shorten <= -2pt]
    \node at (-5,1) {$c=s_1s_2s_3s_4\in W(\widetilde{A}_3)$};
    \node (s4) at (0,0) {$s_4$};
    \node (s1) at (0,2) {$s_1$};
    \node (s3) at (2,0) {$s_3$};
    \node (s2) at (2,2) {$s_2$};
    \draw (s1) -- (s2) -- (s3) -- (s4) -- (s1);
    \node at (1,1) {$\Gamma$};
    \begin{scope}[shift={(5,0)}]
      \node (s4) at (0,0) {$s_4$};
      \node (s1) at (0,2) {$s_1$};
      \node (s3) at (2,0) {$s_3$};
      \node (s2) at (2,2) {$s_2$};
      \draw[to] (s1) -- (s2);
      \draw[to] (s2) -- (s3);
      \draw[to] (s3) -- (s4);
      \draw[to] (s1) -- (s4);
      \node at (1,1) {$\omega(c)$};
    \end{scope}
  \end{tikzpicture}
\end{center}
The conjugate Coxeter elements to $c\in W$ are thus $s_1s_2s_3s_4$,
$s_2s_3s_4s_1$, $s_3s_4s_1s_2$, and $s_4s_1s_2s_3$. The toric filters
of $P(G,[\omega(c)])$ describe which subwords can appear in an initial
segment of some reduced expression of one of these conjugate Coxeter
elements. The poset of these toric filters was shown in
Figure~\ref{fig:C4-1234-ideals} (replace $k$ with $s_k$).

Now, consider the element $c'=s_1s_3s_2s_4$ in $W(\widetilde{A}_3)$,
as shown below:
\begin{center}
  \tikzstyle{to} = [draw,-stealth]
  \tikzstyle{e} = [draw]
  \begin{tikzpicture}[scale=.7,shorten >= -2pt, shorten <= -2pt]
    \node at (-5,1) {$c'=s_1s_3s_2s_4\in W(\widetilde{A}_3)$};
    \node (s4) at (0,0) {$s_4$};
    \node (s1) at (0,2) {$s_1$};
    \node (s3) at (2,0) {$s_3$};
    \node (s2) at (2,2) {$s_2$};
    \draw (s1) -- (s2) -- (s3) -- (s4) -- (s1);
    \node at (1,1) {$\Gamma$};
    \begin{scope}[shift={(5,0)}]
      \node (s4) at (0,0) {$s_4$};
      \node (s1) at (0,2) {$s_1$};
      \node (s3) at (2,0) {$s_3$};
      \node (s2) at (2,2) {$s_2$};
      \draw[to] (s1) -- (s2);
      \draw[to] (s3) -- (s2);
      \draw[to] (s3) -- (s4);
      \draw[to] (s1) -- (s4);
      \node at (1,1) {$\omega(c')$};
    \end{scope}
  \end{tikzpicture}
\end{center}
The toric equivalence class containing $\omega(c')$ has six
orientations, which were shown in Figure~\ref{fig:C4-1324}. Each of
these describes a unique conjugate Coxeter element:
\[
\begin{array}{rrrrrr}
  s_1s_2s_4s_3\,\quad\;\;& s_2s_4s_1s_3\,\quad\;\;& s_4s_1s_3s_2\,\quad\;\;&
  s_2s_1s_3s_4\,\quad\;\;& s_1s_3s_2s_4\,\quad\;\;& s_3s_2s_4s_1\, \\
  =s_1s_4s_2s_3\,\quad\;\;& =s_2s_4s_3s_1\,\quad\;\;& =s_4s_3s_1s_2\,\quad\;\;&
  =s_2s_3s_1s_4\,\quad\;\;& =s_1s_3s_4s_2\,\quad\;\;& =s_3s_4s_2s_1\, \\
                 \quad\;\;& =s_4s_2s_3s_1\,\quad\;\;& &
  & =s_3s_1s_4s_2\,\quad\;\;& \\
                 \quad\;\;& =s_4s_2s_1s_3\,\quad\;\;& &
  & =s_3s_1s_2s_4\,\quad\;\;& \\
\end{array}
\]
These are listed above so that the Coxeter element in the $i^{\rm th}$
column corresponds to the $i^{\rm th}$ orientation in
Figure~\ref{fig:C4-1324}. The linear extensions of each orientation
describe the reduced expressions of the corresponding Coxeter element,
which are listed in the same column above.

The toric poset $P(G,[\omega(c')])$ has four total toric extensions,
and these were shown on the right in Figure~\ref{fig:C4-1324-ideals}
(replace $k$ with $s_k$). The toric filters of $P(G,[\omega(c')])$
correspond to the subsets that appear consecutively in one of these
cyclic words. The poset $J_{\tor}(P(C_4,[\omega']))$ of toric filter
appears on the left in Figure~\ref{fig:C4-1324-ideals}.

%%=====================================================================
\section{Concluding remarks}\label{sec:conclusion}
%%=====================================================================

In this paper, we further developed the theory of toric posets by formalizing the notion of toric intervals, morphisms, and order ideals. In some regards, much of the theory is fairly analogous to that of ordinary posets, though there are some noticeable differences. Generally speaking, the one recurring theme was the characterization of the toric analogue of a feature in $P(G,[\omega])$ by the characterization of the ordinary version of that feature in $P(G,\omega')$ either for some $\omega'\in[\omega]$, or for all $\omega'\in[\omega]$. 

One question that arises immediately is whether there is a toric order complex. While there may exist such an object, there are some difficulties unique to the toric case. For example, a poset is completely determined by its chains, in that if one specifies which subsets of $V$ are the chains of $P$, and then the toric order of the elements within each chain, the entire poset can be reconstructed. This is not the case for toric posets, as shown in Figure~\ref{fig:C5}. Here, two torically non-equivalent orientations of $C_5$ are given, but the toric posets $P(C_5,[\omega])$ and $P(C_5,[\omega'])$ have the same sets of toric chains: the $5$ vertices and the $5$ edges. 

\begin{figure} \centering
  \tikzstyle{to} = [draw,-stealth]
  \tikzstyle{To} = [draw,-stealth,dashed]
  \begin{tikzpicture}[scale=.5,shorten >= -2pt, shorten <= -2pt]
    \node (4) at (0,6) {$\mathbf{4}$};
    \node (3) at (0,4) {$\mathbf{3}$};
    \node (2) at (0,2) {$\mathbf{2}$};
    \node (1) at (0,0) {$\mathbf{1}$};
    \node (5) at (2,2) {$\mathbf{5}$};
    \draw[to] (1) to (2);
    \draw[to] (2) to (3);
    \draw[to] (3) to (4);
    \draw[to] (1) to (5);
    \draw[to] (5) to (4);
    \node at (3.75,3) {$\not\equiv$};
    \node at (-2,5) {$\omega$};
    \begin{scope}[shift={(5.5,0)}]
      \node (1) at (2,6) {$\mathbf{3}$};
      \node (2) at (2,4) {$\mathbf{4}$};
      \node (3) at (2,2) {$\mathbf{5}$};
      \node (4) at (2,0) {$\mathbf{1}$};
      \node (5) at (0,2) {$\mathbf{2}$};
      \draw[to] (2) to (1);
      \draw[to] (3) to (2);
      \draw[to] (4) to (3);
      \draw[to] (5) to (1);
      \draw[to] (4) to (5);
      \node at (4,5) {$\omega'$};
    \end{scope}
  \end{tikzpicture}
  \caption{Two non-torically equivalent orientations
    $\omega\not\equiv\omega'$ in $\Acyc(C_5)$ for which
    $P(C_5,[\omega])$ and $P(C_5,[\omega'])$ have the same set of
    toric chains.}\label{fig:C5}
\end{figure}
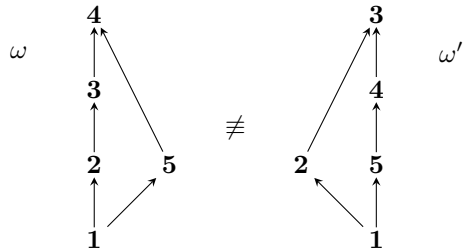

The fact that an ordinary poset is determined by its chains just means
that once one specifies the total order between every chain of size
$k\geq 2$, then the entire partial order is determined. The problem
for toric posets, which we encountered in this paper, is that \emph{every} size-$2$ subset is trivially
cyclically ordered, whether it lies on a toric chain or not. In other
words, a total order can be defined on two elements, but a cyclic
order needs three. The analogous statement for toric posets would be
that specifying the total cyclic order between every toric chain of
size $k\geq 3$ specifies the entire toric order. Such a statement
would establish the intuitive idea that knowing all total cyclic
orders should determine the toric partial order ``modulo the size-$2$
toric chains.'' Current works suggests that there is an analogue of the aforementioned properties for toric posets, but it requires a new generalization of the concept of a chain. The details are too preliminary and complicated to describe here, and it is not clear whether it will lead to a combinatorial object such as a toric order complex. Without this, there might not be a natural way to study toric posets topologically. 

Another important feature of ordinary posets that does not seem to have any obvious toric analogue are M\"obius functions, and this is vital to much of the theory of ordinary posets. Recall the analogy from the Introduction about how topology is like ``analysis without the metric.'' Similarly, many of the basic features of ordinary posets have toric analogues, despite the fact that toric posets have no binary relation. However, much of the more advanced theory is likely to fail because one also seems to lose valuable tools such as an order complex and a M\"obius function. Even the theory that does carry over has its shortcomings. For example, morphisms have a simple combinatorial characterization using the binary relation: $i<_P j$ implies $\phi(i)<_{P'}\phi(j)$. The geometric definition requires a patchwork of quotients, extensions, and inclusions. It would be desirable to have a more ``holistic'' characterization of toric poset morphisms, though it is not clear that that such a description should exist.

Finally, the connection of toric posets to Coxeter groups is the
subject of a paper nearing completion on cyclic reducibility and
conjugacy in Coxeter groups. Loosely speaking, reduced expressions can
be formalized as labeled posets called heaps. This was formalized by
Stembridge~\cite{Stembridge:96,Stembridge:98} in the 1990s. The
\emph{fully commutative} (FC) elements are those such that ``long
braid relations'' (e.g., $sts\mapsto tst$) do not arise. Equivalently,
they have a unique heap. The cyclic version of the FC elements are the
\emph{cyclically fully commutative} CFC elements, introduced by the
author and collaborators in~\cite{Boothby:12}. In 2013, T.~Marquis
showed that two CFC elements are conjugate if and only if their heaps
are torically equivalent~\cite{Marquis:14}. These elements were
further studied by M.~P\'etr\'eolle~\cite{Petreolle:14}. In our
forthcoming paper, we will formalize the notion of a toric heap, which
will essentially be a labeled toric poset. This allows us to formalize
objects such as cyclic words, cyclic commutativity classes, and
develop a theory of cyclic reducibility in Coxeter groups using the
toric heap framework.

\bibliographystyle{alpha} 
%\bibliography{../../../../bibliography/coxeter}

\begin{thebibliography}{BBE{\etalchar{+}}12}

\bibitem[BBE{\etalchar{+}}12]{Boothby:12}
T.~Boothby, J.~Burkert, M.~Eichwald, D.C. Ernst, R.M. Green, and M.~Macauley.
\newblock On the cyclically fully commuative elements of {C}oxeter groups.
\newblock {\em J. Algebraic Combin.}, 36(1):123--148, 2012.

\bibitem[Che10]{Chen:10}
B.~Chen.
\newblock Orientations, lattice polytopes, and group arrangements {I},
  chromatic and tension polynomials of graphs.
\newblock {\em Ann. Comb.}, 13(4):425--452, 2010.

\bibitem[DMR15]{Develin:15}
M.~Develin, M.~Macauley, and V.~Reiner.
\newblock Toric partial orders, 2015.
\newblock Trans. Amer. Math. Soc. In press. arXiv:1211.4247v1.

\bibitem[EE09]{Eriksson:09}
H.~Eriksson and K.~Eriksson.
\newblock Conjugacy of {C}oxeter elements.
\newblock {\em Electron. J. Combin.}, 16(2):\#R4, 2009.

\bibitem[EJ85]{Edelman:85}
P.H. Edelman and R.E. Jamison.
\newblock The theory of convex geometries.
\newblock {\em Geom. Dedicata}, 19(3):247--270, 1985.

\bibitem[Eri94]{ErikssonK:94}
K.~Eriksson.
\newblock Node firing games on graphs.
\newblock {\em Contemporary Mathematics}, 178:117--127, 1994.

\bibitem[ES09]{Ehrenborg:09b}
R.~Ehrenborg and M.~Slone.
\newblock A geometric approach to acyclic orientations.
\newblock {\em Order}, 26:283--288, 2009.

\bibitem[Gei81]{Geissinger:81}
L.~Geissinger.
\newblock The face structure of a poset polytope.
\newblock In {\em Proceedings of the Third Caribbean Conference on
  Combinatorics and Computing}, 1981.

\bibitem[Gre77]{greene1977acyclic}
C.~Greene.
\newblock Acyclic orientations.
\newblock In M.~Aigner, editor, {\em Higher Combinatorics}, pages 65--68,
  Dordrecht, 1977. D. Reidel.

\bibitem[Mar14]{Marquis:14}
T.~Marquis.
\newblock Conjugacy classes and straight elements in coxeter groups.
\newblock {\em J. Algebra}, 407:68--80, 2014.

\bibitem[MRZ03]{Marsh:03}
R.~Marsh, M.~Reineke, and A.~Zelevinsky.
\newblock Generalized associahedra via quiver representations.
\newblock {\em Trans. Amer. Math. Soc.}, 355:4171--4186, 2003.

\bibitem[P\'et14]{Petreolle:14}
M.~P\'etr\'eolle.
\newblock Characterization of cyclically fully commutative elements in finite
  and affine {C}oxeter groups, 2014.
\newblock math.CO/1403.1130.

\bibitem[Pro93]{Propp:93}
J.~Propp.
\newblock Lattice structure for orientations of graphs, 1993.
\newblock Preprint. math.CO/020905.

\bibitem[PRW08]{Postnikov:08}
A.~Postnikov, V.~Reiner, and L.~Williams.
\newblock Faces of generalized permutohedra.
\newblock {\em Doc. Math.}, 13:207 -- 273, 2008.

\bibitem[Shi01]{Shi:01}
J.-Y. Shi.
\newblock Conjugacy relation on {C}oxeter elements.
\newblock {\em Adv. Math.}, 161:1--19, 2001.

\bibitem[Spe09]{Speyer:09}
D.E. Speyer.
\newblock Powers of {C}oxeter elements in infinite groups are reduced.
\newblock {\em Proc. Amer. Math. Soc.}, 137:1295--1302, 2009.

\bibitem[Sta86]{Stanley:86}
R.P. Stanley.
\newblock Two poset polytopes.
\newblock {\em Discrete \& Computational Geometry}, 1(1):9--23, 1986.

\bibitem[Sta01]{Stanley:01}
R.P. Stanley.
\newblock {\em Enumerative Combinatorics: Volume 2}.
\newblock Cambridge University Press, 2001.

\bibitem[Ste96]{Stembridge:96}
J.R. Stembridge.
\newblock On the fully commutative elements of {C}oxeter groups.
\newblock {\em J. Algebraic Combin.}, 5:353--385, 1996.

\bibitem[Ste98]{Stembridge:98}
J.R. Stembridge.
\newblock The enumeration of fully commutative elements of {C}oxeter groups.
\newblock {\em J. Algebraic Combin.}, 7(3):291--320, 1998.

\bibitem[Stu96]{sturmfels1996grobner}
B.~Sturmfels.
\newblock {\em Gr\"obner bases and convex polytopes}, volume~8 of {\em
  University Lecture Series}.
\newblock American Mathematical Society, Providence, RI, 1996.

\bibitem[Wac07]{Wachs:07}
M.L. Wachs.
\newblock Poset topology: Tools and applications.
\newblock In {\em Geometric Combinatorics, IAS/Park City Mathematics Series},
  pages 497--615, 2007.

\bibitem[War42]{ward1942closure}
M.~Ward.
\newblock The closure operators of a lattice.
\newblock {\em Ann. Math.}, pages 191--196, 1942.

\bibitem[Zas91]{zaslavsky1991orientation}
T.~Zaslavsky.
\newblock Orientation of signed graphs.
\newblock {\em European J. Combin.}, 12(4):361--375, 1991.

\end{thebibliography}

\newcommand{\etalchar}[1]{$^{#1}$}

\end{document}